\documentclass{article}

\usepackage[utf8]{inputenc}
\usepackage{graphicx}
\usepackage{amsthm}
\usepackage{amsmath}
\usepackage{amssymb}
\usepackage{authblk}
\usepackage{xcolor}
\usepackage{verbatim}
\usepackage{subcaption}
\usepackage{amsthm}
\usepackage{tikz}
\usepackage{float}
\usepackage[authoryear,round]{natbib}
\newcommand{\R}{{\mathbb R}}
\newcommand{\EE}{{\mathbb E}}
\newcommand{\PP}{{\mathbb P}}
\newcommand{\N}{{\mathbb N}}

\newcommand{\cX}{{\cal{X}}}
\newcommand{\cN}{{\cal{N}}}

\newcommand{\xx}{{\mathbf{x}}}
\newcommand{\yy}{{\mathbf{y}}}

\newcommand{\uu}{{\mathbf{u}}}

\newcommand{\1}{{\mathbf{1}}}

\newtheorem{theorem}{Theorem}[section]
\newtheorem{proposition}[theorem]{Proposition}
\newtheorem{corollary}{Corollary}[theorem]

\newtheorem{algorithm}[theorem]{Algorithm}

\usepackage[margin=1in]{geometry}
\counterwithin{equation}{section}

\title{State estimation for aoristic models}

\author[1,2]{M.N.M.\ van Lieshout}
\author[1,2]{R.L.\ Markwitz}
\affil[1]{Centrum Wiskunde \& Informatica (CWI), P.O. Box 94079, 
1090 GB Amsterdam, The Netherlands}
\affil[2]{Department of Applied Mathematics, University of Twente, 
P.O. Box 217, 7500 AE Enschede, The Netherlands}
\date{}

\begin{document}

\maketitle

\noindent
ABSTRACT: Aoristic data can be
described by a marked point process in time in which the points cannot
be observed directly but are known to lie in observable intervals, the
marks. We consider Bayesian state estimation for the latent points when
the marks are modelled in terms of an alternating renewal process in
equilibrium and the prior is a Markov point point process. We derive the
posterior distribution, estimate its parameters and present some examples
that illustrate the influence of the prior distribution.\\[0.1in]

\noindent
KEYWORDS: alternating renewal process, aoristic data, marked temporal
point process, Markov chain Monte Carlo methods, Markov point process
state estimation.

\section{Introduction}

Inference for point processes on the real line has been dominated
by a dynamic approach based on the stochastic intensity or hazard
function \citep{Brem72,Karr91,Last95}. Such an approach is quite
natural, is amenable to likelihood-based inference and allows the
utilisation of powerful tools from martingale theory. However, it
breaks down completely when censoring breaks the orderly
progression of time. In such cases, state estimation techniques
are needed that are able to fill in the gaps \citep{BrixDigg01,Lies16}.

In this paper, we concentrate on aoristic data \citep{RatcMcCu98} in
which the points may not be observed directly but upper and lower
bounds exist. Such data are commonplace in criminology. Suppose, for
example, that a working person leaves his place of residence early in
the morning and returns late in the afternoon to discover that the
residence has been burgled. Then the exact time of the break-in
cannot be determined, but it must have happened during the absence
of the resident. In rare cases, a burglar may also be caught in the
act, in which case the time of break-in coincides with the bounds.
The actual break-in times tend to be estimated by ad hoc, naive
approaches, e.g.\ the mid-point of the reported interval \citep{Helm08}
or the length-weighted empirical probability mass function of the
interval lengths. An obvious disadvantage of such methods is that
dependencies between offence times, such as the near-repeat effect
\citep{Bern09}, are ignored. 

The focus of this research is to develop a Bayesian inference
framework for aoristic data that is able to infer missing 
information and takes into account expert knowledge and 
interaction. Specifically, in Section~\ref{sec:renewal} we 
formalise aoristic censoring as a marked point process in time 
in which the points cannot be observed directly but are known 
to lie in observable intervals, the marks. Upon employing a 
Markov point process prior \citep{Lies00}, the posterior 
distribution of the point locations is derived in 
Section~\ref{sec:posterior}. In Section~\ref{sec:mcmc} we turn to 
Monte Carlo based inference. The paper is concluded by 
simulated examples that demonstrate the influence of the prior
(Section~\ref{sec:sim}) and some reflections on future research.

\section{Marked point process formulation}
\label{sec:renewal}

\subsection{Alternating renewal processes for censoring}
\label{sec:censoring}

In this paper, we consider a censoring mechanism based on an
alternating renewal process. Let $C_1, C_2,\,...$ be a sequence of
random $2$-vectors such that $C_i = (Y_i, Z_i)$, $i \in \N$, are
independent and identically distributed \citep{Asmu03,Ross96}. 
Furthermore, assume that $C_i$ has joint probability density
function $f$ on $(\R^+)^2$.  Introduce $T_i$ = $Y_i + Z_i$, set
$S_0 = 0$ and let, for $n\in \N$, $S_n = \sum_{i=1}^n T_i$
be the time of the $n$th renewal. Note that no renewal occurs at
the end of a $Y$-phase. Furthermore, assume that $0 < \EE T_1
< \infty$. Then, by the strong law of large numbers, 
\begin{equation}
N(t) =  \sup \left\{ n\in\N_0 : S_n \leq t \right\}, 
\quad t \geq 0,
\label{e:Nt}
\end{equation}
is well-defined and the supremum is attained with probability one.
Furthermore, the renewal function 
\[
M(t) = \EE N(t) = \sum_{n=1}^\infty \PP(S_n \leq t), \quad t \geq 0,
\]
is finite and absolutely continuous with respect to Lebesgue measure 
\citep[Chapter~3]{Ross96}.
\begin{figure}
    \centering
    \begin{tikzpicture}
    \draw [<->, thick] (-5,0) -- (5,0);
    \draw [ultra thick] (1, -.15) -- (1, .15);
    \filldraw [black] (-3,0) circle (2pt) node [above] {$S_{i}$};
    \filldraw [black] (2.6,0) circle (2pt) node [above] {$S_{i+1}$};
    \draw [ultra thick] (-3,0) -- (1,0); 
    \node at (-1, -.5) {$Y_{i + 1}$};
    \node at (1.8, -.5) {$Z_{i + 1}$};
    \node at (-5.5, 0) {$(1)$};
    \node at (-1, .3) {$t$};
    \draw [<->, thick] (-5,-2) -- (5,-2);
    \filldraw [ultra thick] (0.4,-2.15) -- (0.4, -1.85);
    \filldraw [black] (-3.4,-2) circle (2pt) node [above] {$S_{i}$};
    \filldraw [black] (3,-2) circle (2pt) node [above] {$S_{i+1}$};
    \draw [thick] (1,-2) -- (1, -1.8) node [above] {$t$}; 
    \node at (-1.4, -2.5) {$Y_{i + 1}$};
    \node at (1.7, -2.5) {$Z_{i + 1}$};
    \node at (-5.5, -2) {$(2)$};
    \end{tikzpicture}
\caption{View of two alternating renewal processes. In $(1)$, a point
  $t$ of the point process $X$ falls in a $Y$-phase. In criminological
  terms, this may correspond to a person being away from home and
  being burgled; the exact time of burglary is unknown. Hence the
  whole interval is recorded. In $(2)$, a point $t$ of the point
  process $X$ falls in a $Z$-phase. This corresponds to a person
  being home and being burgled. In this case, the exact time of
  burglary is known, thus the exact point is recorded. Renewal times
  are denoted by $S_i,\, S_{i+1}$ with $S_0 = 0$.}
 \label{fig:arp}
\end{figure}
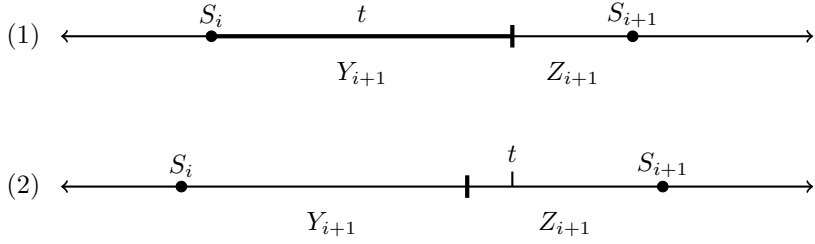
\\ \\
An alternating renewal process can be used for censoring in the 
following way. Let $X$ be a temporal point process on $\R^+$ 
\citep{Brem72,DaleVere88,Last95} and let each point $t\in X$ be
associated with an alternating renewal process, independently of
other points of $X$. Now, if $t$ happens to fall in some $Z$-phase, 
$t$ is observed perfectly, whereas $t$ is observed aoristically if 
it falls in a $Y$-phase. The censoring mechanism is illustrated 
in Figure~\ref{fig:arp}.

\subsection{Age and excess distribution}
\label{sec:markdistribution}

Aoristic data generated by the censoring mechanism described
in Section~\ref{sec:censoring} can be expressed in terms of the age 
and excess (also referred to as residual lifetime) with respect to the 
$Y$-process. Recall that for an alternating renewal process, the age 
with respect to the $Y$-process is defined as
\begin{align*}
    A(t) &= (t - S_{N(t)}) \, \1\{S_{N(t)}+Y_{N(t)+1} > t\}, 
\quad t \geq 0,
\end{align*}
the excess with respect to the $Y$-process as
\begin{align*}
    B(t) &= (S_{N(t) + 1} - Z_{N(t) + 1} - t)\,
     \1\{S_{N(t)}+Y_{N(t)+1} > t\}, \quad t \geq 0.
\end{align*}

Let us parametrise an interval on the real line by its left-most
point and length. In other words, a pair $(a, l) \in \R \times \R^+$
corresponds to the closed interval $[a, a+l]$. Then, the recorded
interval of a latent point $t\geq 0$ can be written as $t + [-A(t), 
B(t)]$. Note that $[-A(t), B(t)]$ is parametrised by 
$I(t) = (- A(t), A(t) + B(t) ) \in \R \times \R^+$. 

First, let us consider the joint distribution of age and excess with
respect to the $Y$-process.

\begin{proposition}
Let $N$ be an alternating renewal process as in (\ref{e:Nt}). 
Assume that $T_1$ is absolutely continuous with respect to Lebesgue 
measure and that $0 < \EE T_1 < \infty$. Then, for $t\geq 0$,
the joint distribution of $(A(t), B(t))$ has an atom at $(0,0)$ of size 
\[
c(t) = 
F_Y(t) - \int_0^t \left[ 1 - F_Y(t-s) \right] dM(s)
\]
and, for $0 \leq u \leq t, v \geq 0$,
\[
\PP( A(t) \leq u; B(t) \leq v ) = c(t) + 
\left[ F_Y(t+v) - F_Y(t) \right] \1\{ u = t \}
+ \int_{t-u}^t \left[
 F_Y( t+v-s) - F_Y(t-s)
\right] dM(s).
\]
Here $F_Y$ denotes the cumulative distribution function of $Y_1$
and $M$ is the renewal function.
\label{p:jointAB}
\end{proposition}

\begin{proof}[Proof.]
Write $F_n$ for the cumulative distribution function of $S_n$, $n\in \N$.
By partitioning over the number of renewals up to time $t$ and upon
noting that $N(t) = n$ if and only if $S_n \leq t$ and 
$S_n + Y_{n+1} + Z_{n+1} > t$, one obtains that
\begin{eqnarray*}
\PP( A(t) \leq u; B(t) \leq v ) & = & c(t) + 
\PP( t - S_{N(t)} \leq u; \, t < S_{N(t)} + Y_{N(t)+1} \leq t + v ) \\
& = & c(t) + \sum_{n=0}^\infty 
\PP( t - u \leq S_{N(t)}; \, t < S_{N(t)} + Y_{N(t)+1} \leq t + v ; 
\, N(t) = n ) \\
& = & 
c(t) + \PP( t - u \leq S_{N(t)}; \, t < S_{N(t)} + Y_{N(t)+1} \leq t + v; 
\, N(t) = 0 ) \\
& + & \sum_{n=1}^\infty \int_{t-u}^t 
   \PP( t-s < Y_{n+1} \leq t+v-s ) \, dF_n(s).
\end{eqnarray*}
The claim follows by an application of
Fubini's theorem for the last term, the observation that
\[
 \PP( t - u \leq S_{N(t)}; \, t < S_{N(t)} + Y_{N(t)+1} \leq t + v ; 
\, N(t) = 0 ) =
 \PP( t < Y_1 \leq t+v )
\]
if $u=t$ and zero otherwise, and because 
\begin{eqnarray*}
c(t) & = & 1 - \PP( S_{N(t)} + Y_{N(t)+1} > t ) \\
& = & 1 - \PP( Y_1 > t )  - \sum_{n=1}^\infty \int_0^t  
      \PP( Y_{n+1} > t-s ) \, dF_n(s) 
 =  F_Y(t) - \int_0^t \left[ 1 - F_Y(t-s) \right] dM(s).
\end{eqnarray*}
\end{proof}

The long-run behaviour as time goes to infinity can be obtained
by appealing to the key renewal theorem (for example found in
\citet[Theorem~3.4.2]{Ross96}). Specialising to the parameter
vector $I(t)$, the following theorem holds. 

\begin{theorem}
\label{t:nu}
Let $N$ be an alternating renewal process as in (\ref{e:Nt}). 
Assume that $Y_1$ and $T_1$ are absolutely continuous with respect to Lebesgue 
measure with probability density functions $f_Y$ and $f$ and that 
$0 < \EE T_1 < \infty$. Then $(-A(t), A(t)+B(t))$ tends in
distribution to $\nu$, the mixture of an atom at $(0,0)$ and an 
absolutely continuous component that has probability density function
\(
f_Y(l) / \EE Y_1 
\)
on $\{ (a,l) \in \R \times \R^+ : a \leq 0 \leq a + l \}$. The mixture
weights are, respectively, $\EE Z_1 / \EE T_1$ and $\EE Y_1 / \EE T_1$.
\end{theorem}

\begin{proof}[Proof.]
First, let us consider the limit behaviour of the joint cumulative
distribution function of $A(t)$ and $B(t)$ as $t\to\infty$.
With the notation of Proposition~\ref{p:jointAB}, by Theorem~3.4.4 
in \citep{Ross96}, $c(t)$ converges to $\EE Z_1 / \EE T_1$. Also, 
for $t > u$, the second term in the joint cumulative distribution 
function of $A(t)$ and $B(t)$ is zero. For the last term, note that 
for $v\geq 0$ the function $h_v: \R^+ \to \R$ defined by 
$h_v(s) = 1 - F_Y(v+s)$ is non-negative, monotonically non-increasing 
and integrable. Hence the key renewal theorem implies that 
\[
\lim_{t\to\infty} \int_0^t \left[ 1 - F_Y(t+v-s) \right] dM(s)
= \frac{1}{\EE T_1} \int_0^\infty \left[ 1 - F_Y(v+s) \right] ds.
\]
Analogously, for fixed $u\geq 0$, $t-u \to \infty$ if and only
if $t\to \infty$. Writing $s = t-u$,
\[
\lim_{s\to\infty} \int_0^s \left[ 1 - F_Y(s+u+v-r) \right] dM(r)
= \frac{1}{\EE T_1} \int_0^\infty \left[ 1 - F_Y( u+v+r) \right] dr
\]
\[
= \lim_{t\to\infty} \int_0^{t-u} \left[ 1 - F_Y(t+v-r) \right] dM(r).
\]
We conclude that 
\(
Q(u,v) =
\lim_{t\to\infty} \PP( A(t) \leq u; B(t) \leq v ) 
\)
exists and equals
\[
Q(u,v)
= 
\frac{\EE Z_1}{\EE T_1} + 
\frac{1}{\EE T_1} \int_0^u \left[ F_Y(v+s) - F_Y(s) \right] ds.
\]
Note that $Q$ is the cumulative distribution function of
the mixture of an atom at $(0, 0)$ and an absolutely continuous
component with probability density function
\(
 f_Y(u+v) / \EE Y_1
\)
on $(\R^+)^2$. By Helly's continuity theorem, $(A(t), B(t))$
converges in distribution.

Turning to the parametrisation $I(t) = (-A(t), A(t) + B(t))$,
its limit distribution inherits an atom at $(0,0)$ from $Q$
of size $\EE Z_1 / \EE T_1$. By the change of variables bijection
$h: (\R^+)^2 \to \R^- \times \R^+$ defined by $h(u, v) = (-u,
u+v)$, since $h$ is differentiable, the absolutely continuous 
part has probability density function
$f_Y(h^{-1}(a,l))\,|\text{det}\,J_{h^{-1}}(a,l)| / \EE Y_1 = 
f_Y(-a+a+l) / \EE Y_1$, where $J_{h^{-1}}$ is the Jacobian of 
$h^{-1}$. 
\end{proof}

\subsection{Complete model formulation}
\label{sec:model}

We are now ready to formulate a model. Let $\cX$ be an open subset 
of the positive half-line $\R^+$. The state space of $X$, denoted 
by $\cN_\cX$, consists of finite sets $\{ t_1, \dots, t_n \} \subset  \cX$, 
$n\in\N_0$, which we equip with the Borel $\sigma$-algebra of the weak 
topology \citep{DaleVere88}. We will assume that the distribution of $X$ is 
specified in terms of a probability density function $p_X$ with respect to 
the distribution of a unit rate Poisson process on $\cX$ \citep{Lies00}.

Upon labelling the points of $X$ independently with a mark according 
to the mixture distribution of Theorem~\ref{t:nu}, denoted by $\nu$,
the complete model $W$ is obtained. Its realisations are sets
$\{ (t_1, I_1), \dots, (t_n, I_n) \} \subset \cX \times (\R \times
\R^+ )$. For parametrisation $I_j = (a_j, l_j)$, the pair $(t_j, I_j)$
defines an interval $[ t_j + a_j, t_j + a_j + l_j ]$. The ensemble
of all realisations is denoted by $\cN_{\cX \times (\R \times \R^+)}$
and equipped with the Borel $\sigma$-algebra of the weak topology. 
Note that $W$ has probability density function $p_X$ with respect to the 
distribution of a Poisson process on $\cX \times (\R \times \R^+)$ 
with intensity measure $\ell \times \nu$ where $\ell$ is Lebesgue
measure.

Due to the censoring, one does not observe the complete model $W$
but rather the set
\begin{equation}
\label{eq:observeU}
U = \bigcup_{(t,I)\in W} ( (t,0) + I ).
\end{equation}
Our aim is to reconstruct $X$ or $W$ from $U$. In order to do
so, the posterior distribution of $X$ or $W$ given $U$ is needed.
This will be the topic of the next section.

\section{The Bayesian framework}
\label{sec:posterior}

In a Bayesian framework, the posterior distribution updates prior forms 
in the light of data gathered \citep{Game06}. Heuristically,
\begin{align}
p_{X|U}(\mathbf{x}\,|\,\mathbf{u}) 
\propto 
p_{U|X}(\mathbf{u}\,|\,\mathbf{x})\,\,p_{X}(\mathbf{x})
\label{eq:bayes_theorem}
\end{align}
through the use of Bayes' theorem. The term $p_{U|X}(\uu\,|\,\xx)$ 
describes the likelihood that the points of $\xx$ generate the intervals 
in $\uu$. In the literature this term is referred to as a forward term, 
forward density or forward model \citep{Lies95,LiesBadd02}. The term 
$p_X(\xx)$ captures prior beliefs about the geometry of $\xx$.
In our context, since the forward model is a mixture of
discrete and absolutely continuous components,  some care is
required in handling (\ref{eq:bayes_theorem}).

\begin{theorem}
Let $W$ be a point process on the open set $\cX \subset \R$ 
with probability density function $p_X$ with respect to the distribution 
of a unit rate Poisson process on $\cX$ marked independently
with mark distribution $\nu$ defined in Theorem~\ref{t:nu} .
Write $X$ for the ground process of locations in $\cX$ and 
consider the forward model (\ref{eq:observeU}). Let $\uu$
be a realisation of $U$ that consists of an atomic part 
$\{ (a_1, 0), \dots, (a_m, 0) \}$, $m\in\N_0$, and a non-atomic
part $\{ (a_{m+1}, l_{m+1}),\, \dots, (a_n, l_n) \}$, $n \geq m$.
Then the posterior distribution of $X$ given $U = \uu$ satisfies, for 
$A$ in the Borel $\sigma$-algebra of the weak topology on $\cN_\cX$,
\[
\PP( X \in A \mid U = \uu ) = c(\uu) \int_{\cX^{n-m}}
p_X( \{ a_1, \dots, a_m, x_1, \dots, x_{n-m} \}) 
\1_A(\{ a_1, \dots, a_m, x_1, \dots, x_{n-m} \} ) \times
\]
\[
\times \left(
\sum_{ \substack{D_1, \dots, D_{n-m}   \\
       \cup_i \{ D_i \} = \{ 1, \dots, n-m \} } }
\prod_{i=1}^{n-m} \1\{ x_{D_i} \in [a_{m+i}, a_{m+i}+l_{m+i}] \}
\right) \prod_{i=1}^{n-m} dx_i
\]
provided that  $c(\uu)^{-1}$ defined by
\[
\int_{\cX^{n-m}}
p_X( \{ a_1, \dots, a_m, x_1, \dots, x_{n-m} \}) 
\left(
\sum_{ \substack{D_1, \dots, D_{n-m}   \\
       \cup_i \{ D_i \} = \{ 1, \dots, n-m \} } }
\prod_{i=1}^{n-m} \1\{ x_{D_i} \in [a_{m+i}, a_{m+i}+l_{m+i}] \}
\right) \prod_{i=1}^{n-m} dx_i
\]
exists in $(0,\infty)$.
\label{prop:cond_prob_viable}
\end{theorem}

\begin{proof}[Proof.]
We must show that for each $A$ in the Borel $\sigma$-algebra of
$\cN_\cX$ with respect to the weak topology and each $F$ in the 
Borel $\sigma$-algebra of the weak topology on $\cN_{\R\times\R^+}$
the following identity holds: 
\begin{equation}
\EE \left[ \1_F(U) \, \PP( X \in A \mid U ) \right] = 
\EE \left[ \1_F(U) \, \1_A(X) \right].
\label{eq:cond-expectation}
\end{equation}

Let
\begin{align*}
    q_x(a,l) = \frac{f_Y(l)}{\EE Y_1}\,\1\{a \leq x \leq a+l\}
\end{align*}
describe a probability density function for parametrisations of 
intervals generated 
by $x\in\cX$, noting that it is jointly measurable
as a function on $\cX \times (\R \times \R^+)$. Then, denoting the
cardinality of a set by $|\cdot|$ and Lebesgue measure by $\ell$,
$\EE [ \1_F(U) \, \1_A(X)]$ can be expanded as 
\[
\sum_{n=0}^{\infty}\frac{e^{-\ell(\cX)}}{n!} 
   \sum_{C_0 \subset\{1, ..., n\}} 
   \left( \frac{\EE Z_1}{\EE T_1} \right)^{|C_0|}
   \left( \frac{\EE Y_1}{\EE T_1}  \right)^{n-|C_0|}
   \int_{\cX^n} \1_A( \{ x_1, \dots, x_n \} ) \,
   p_X( \{ x_1, ..., x_n \} ) 
   \sum_{\substack{ C_1, ..., C_{n-|C_0|} \\ 
   \cup_j \{ C_j \} = \{1, ..., n\} \setminus C_0}} \, 
\]
\begin{equation}
   \frac{1}{(n - |C_0|)!} \left(
   \int_{(\R\times \R^+)^{n-|C_0|}}
   \prod_{j=1}^{n-|C_0|} q_{x_{C_j}}(u_j) \,
   \1_F( \{u_1, \dots, u_{n-|C_0|} \} \cup 
     \{ (x_k,0) : k \in C_0 \}  )
   \prod_{j=1}^{n-|C_0|} \,du_j  \right)
   \prod_{i=1}^{n}\,dx_i. 
\label{eq:final_rhs}
\end{equation}
For the left-hand side of (\ref{eq:cond-expectation}), expanding as before, we get
\[
\sum_{n=0}^{\infty}\frac{e^{-\ell(\cX)}}{n!} 
   \sum_{C_0 \subset\{1, ..., n\}} 
   \left( \frac{\EE Z_1}{\EE T_1} \right)^{|C_0|}
   \left( \frac{\EE Y_1}{\EE T_1} \right)^{n-|C_0|}
   \int_{\cX^n} p_X( \{ x_1, ..., x_n \} ) 
  \sum_{\substack{ C_1, ..., C_{n-|C_0|} \\ 
   \cup_j \{ C_j \} = \{1, ..., n\} \setminus C_0}} \, \frac{1}{(n - |C_0|)!}
\]
\[
\left(
   \int_{(\R\times \R^+)^{n-|C_0|}}
   \prod_{j=1}^{n-|C_0|} q_{x_{C_j}}(u_j) \,
   \PP\left(X\in A \mid U = \{ u_1, \dots, u_{n-|C_0|} \} \cup 
      \{ (x_k,0): k \in C_0 \} \right)  \times
\right.
\]
\[
\times \left. 
  \1_F(  \{u_1, \dots, u_{n-|C_0|} \} \cup
     \{ (x_k, 0): k \in C_0 \} )
   \prod_{j=1}^{n-|C_0|} \,du_j \right) \prod_{i=1}^{n}\,dx_i.  
\]
Plugging in the claimed expression for the conditional expectation, one obtains
\[
\sum_{n=0}^{\infty}\frac{e^{-\ell(\cX)}}{n!} 
   \sum_{C_0 \subset\{1, ..., n\}} 
  \left( \frac{\EE Z_1}{\EE T_1} \right)^{|C_0|}
   \left( \frac{\EE Y_1}{\EE T_1} \right)^{n-|C_0|}  
   \int_{\cX^n} p_X( \{ x_1, ..., x_n \} ) 
 \sum_{\substack{ C_1, ..., C_{n-|C_0|} \\ 
   \cup_j \{ C_j \} = \{1, ..., n\} \setminus C_0}} \, \frac{1}{(n - |C_0|)!}
\]
\[
   \int_{(\R\times \R^+)^{n-|C_0|}}
   \prod_{j=1}^{n-|C_0|} q_{x_{C_j}}(u_j) \,  c( \{u_1, \dots, u_{n-|C_0|} \} \cup
     \{ (x_k, 0) : k \in C_0 \} ) 
\]
\[
\left(
   \int_{\cX^{n-|C_0|}} 
     p_X( \{ y_1, \dots, y_{ n-|C_0|} \} \cup \{ x_k: k \in C_0 \} ) \,
\1_A( \{ y_1, \dots, y_{n-|C_0|} \} \cup \{ x_k: k \in C_0 \} ) 
\sum_{\substack{ D_1, ..., D_{n-|C_0|} \\ 
   \cup_j \{ D_j \} = \{1, ..., n\} \setminus C_0}} 
   \right.
   \]
   \[
   \left.
   \prod_{k=1}^{n-|C_0|} \1\{ y_{D_k} \in [ u_{k,1} , u_{k,1} + u_{k,2} ] \} \, dy_k   
\right)
 \1_F( \{u_1, \dots, u_{n-|C_0|} \} \cup
     \{ (x_k, 0) : k \in C_0 \} ) \,   
    \prod_{j=1}^{n-|C_0|} \,du_j \prod_{i=1}^n dx_i.
\]
Note that in order to cancel terms, the order of integration must be changed. 
As evidently the first term in $q_x(a,l)$, that is
${f_Y(l)}  / {\EE Y_1}$, does not depend on $x$, by Fubini's theorem,
\begin{align*}
    \EE[1_F(U) \, \PP(X \in A \mid U)] &= 
    \sum_{n=0}^{\infty} \frac{e^{-\ell(\cX)}}{n!} 
    \sum_{C_0 \subset\{1, ..., n\}}
    \left(\frac{\EE Z_1}{\EE T_1}\right)^{|C_0|} 
    \left(\frac{\EE Y_1}{\EE T_1}\right)^{n-|C_0|}  \\
\end{align*}
\[
  \int_{\cX^n} 
    p_X( \{ y_1, \dots, y_{n-|C_0|} \} \cup \{ x_k: k \in C_0 \} ) \,
    \1_A( \{ y_1, \dots, y_{n-|C_0|} \} \cup \{ x_k: k \in C_0 \} ) 
\]
\[
  \sum_{\substack{ D_1, ..., D_{n-|C_0|} \\ 
   \cup_j \{ D_j \} = \{1, ..., n\} \setminus C_0}} \frac{1}{(n-|C_0|)!}
   \int_{(\R\times \R^+)^{n-|C_0|}}
 \1_F( \{u_1, \dots, u_{n-|C_0|} \} \cup
     \{ (x_k, 0) : k \in C_0 \} )      \prod_{j=1}^{n-|C_0|} q_{y_{D_j}}(u_j)  \,
\]
\[
 \left(
   \int_{\cX^{n-|C_0|}} 
     p_X( \{ x_1, \dots, x_n \} )
    \sum_{\substack{ C_1, ..., C_{n-|C_0|} \\ 
    \cup_j \{ C_j \} = \{1, ..., n\} \setminus C_0}} 
  \prod_{i=1}^{n-|C_0|} \1\{ x_{C_i} \in [ u_{i,1} , u_{i,1} + u_{i,2} ] \} \, dx_{C_i}
  \right) \times
\]
\[
   \times c( \{u_1, \dots, u_{n-|C_0|} \} \cup
      \{ (x_k, 0) : k \in C_0 \}  ) 
 \prod_{j=1}^{n-|C_0|} \,du_j \prod_{k=1}^{n-|C_0|}\,dy_k \prod_{i\in C_0} dx_i  
    = \EE[ \1_F(U) \, \1_A(X)] 
\]
after cancelling and rearranging terms and noting that the term in between
brackets cancels out against the normalisation constant 
$c( \{u_1, \dots, u_{n-|C_0|} \} \cup \{ (x_k,0) : k \in C_0 \})$.
\end{proof}

Theorem~\ref{prop:cond_prob_viable} states that the posterior 
distribution of $X$ given $U = \uu$ is the union of $m$ atoms
combined with $n-m$ points that are distributed on $\cX^{n-m}$
according to the symmetric probability density function
\begin{equation}
\label{eq:jn}
c(\uu) \, p_X( \{ a_1, \dots, a_m, x_1, \dots, x_{n-m} \}) 
\sum_{ \substack{D_1, \dots, D_{n-m}   \\
       \cup_i \{ D_i \} = \{ 1, \dots, n-m \} } }
\prod_{i=1}^{n-m} \1\{ x_{D_i} \in [a_{m+i}, a_{m+i} + l_{m+i}] \}
\end{equation}
with respect to Lebesgue measure. 

\begin{corollary}
In the framework of Theorem~\ref{prop:cond_prob_viable},
the conditional distribution of the mark assignments $D_1, 
\dots, D_{n-m}$ for non-atomic marks is as follows.
For $d_1, \dots d_{n-m} \in \{ 1, \dots, n-m \}$ such that
$\{ d_1, \dots, d_{n-m} \}$ $ = \{ 1, \dots, n-m \}$,
\[
\PP(D_1 = d_1, \dots, D_{n-m} = d_{n-m} | X = 
\{ a_1, \dots, a_m, x_1, \dots, x_{n-m} \}, U = \uu )
\]
\[
= \frac{
\prod_{i=1}^{n-m} \1 \{ x_{d_{i}} \in 
   [ a_{m+i}, a_{m+i}+l_{m+i}] \}
}{
\sum_{ \substack{C_1, \dots, C_{n-m}   \\
       \cup_i \{ C_i \} = \{ 1, \dots, n-m \} } }
\prod_{i=1}^{n-m} \1 \{ x_{C_{i}} \in 
   [ a_{m+i}, a_{m+i}+l_{m+i}] \}
}
\]
provided that $x_i \in [a_{m+i}, a_{m+i} + l_{m+i}]$ for
$i=1, \dots, n-m$ and zero otherwise.
\end{corollary}

As a special case, let us consider an inhomogeneous Poisson process
with integrable intensity function $\lambda: \cX \to \R^+$. Then, under
the posterior distribution, $X$ consists of $n$ independent points, one
in each interval of $\uu$, with probability density function
\[
\frac{\lambda(x)}{\int_{[a_i, a_i+l_i] \cap \cX} 
\lambda(s) ds}
\]
on $[a_i, a_i + l_i] \cap \cX$ for intervals with $l_i>0$.
To see this, recall that for a Poisson process \citep{Lies00}
\[
p_X( \{ a_1, \dots, a_m, x_1, \dots, x_{n-m} \}) =
\exp\left[ \int_{\cX} (1 - \lambda(s) ) ds \right] 
\prod_{j=1}^m \lambda(a_j) \prod_{i=1}^{n-m} \lambda(x_{i})
\]
which factorises over terms associated with each interval. Hence 
(\ref{eq:jn}) is proportional to
\[
\sum_{ \substack{D_1, \dots, D_{n-m}   \\
       \cup_i \{ D_i \} = \{ 1, \dots, n-m \} } }
\prod_{i=1}^{n-m} \lambda(x_{D_i}) \,
   \1 \{ x_{D_i} \in [a_{m+i}, a_{m+i} + l_{m+i}] \}.
\]

\section{Statistical inference}
\label{sec:mcmc}

In this section we will consider statistical inference for
aoristically censored data. Our main aim is to reconstruct
the latent point process $X$ from observed parametrised 
intervals $U$, that may or may not be censored. In tandem, the
censoring probability as well as the parameters $\eta$ of the 
distribution of the non-degenerate intervals must be estimated.
Parameters of the prior distribution may either be treated as fixed or
subject to estimation.


\subsection{Forward model parameters}
\label{sec:forward}

Suppose that we observe a realisation 
\(
\uu = \{ (a_1, 0), \dots, (a_m, 0), 
 (a_{m+1}, l_{m+1}), \dots, (a_n, l_n) \}
\)
of $U$, where $a_i \in \R$, $l_i > 0$ and $n\neq 0$.
Our first aim is to estimate the parameters $\eta$ of the mark distribution $\nu$
(cf.\ Theorem~\ref{t:nu}). The parameter vector $\eta$ comprises the 
parameters $\zeta$ of the probability density function $f_Y$ as well as any other
parameters $\theta$ involved in the joint distribution of the random vector
$C_1 = (Y_1, Z_1)$ that defines the alternating renewal process (cf.\
Section~\ref{sec:censoring}). 

The likelihood function can be obtained from the proof of 
Theorem~\ref{prop:cond_prob_viable} by taking  $A$ equal to 
$\cN_\cX$ in equation (\ref{eq:final_rhs}). On a logarithmic scale,
\begin{equation}
\label{eq:estimateForward}
L(\eta; \uu) = m \log \left(
\frac{\EE [Z_1; \zeta, \theta]}{\EE [T_1;\zeta,\theta]} \right)
+
(n-m) \log\left(
\frac{\EE [Y_1;\zeta]}{\EE T_1[\zeta,\theta]} \right)
+
\sum_{i=1}^{n-m} \log \left( \frac{ f_Y(l_i;\zeta) }{ \EE [Y_1;\zeta]}  \right)
\end{equation}
upon ignoring terms that do not depend on $\eta$. 

Equation (\ref{eq:estimateForward}) simplifies greatly if we assume that
the mixture weight $p = \EE[Z_1; \zeta, \theta] / \EE[T_1; \zeta, \theta]$ does
not depend on $\zeta$. Then $\eta = (p, \zeta)$ and
\[
L(p, \zeta; \uu) = m \log p + (n-m) \log(1-p) + 
\sum_{i=1}^{n-m} \log \left( \frac{ f_Y(l_i;\zeta) }{ \EE [Y_1;\zeta]}  \right).
\]
The atom probability $p$ may be estimated by $m/n$,  the fraction
of atoms in the sample $\uu$. For $\zeta$, we need the following result.

\begin{proposition}
Let  $\nu$ be as in Theorem~\ref{t:nu}. Then the distribution of 
the lengths of non-degenerate intervals is given by the
length-weighted marginal distribution 
$f(l) = {l f_Y(l)} / {\EE Y_1}$ and the left-most points 
are, conditionally on $L = l$, uniformly distributed on $[-l, 0]$.
\end{proposition}

\begin{proof}[Proof.]
Let $f(l)$ be the marginal distribution of the length $l$. Evaluating,
\begin{align*}
    f(l) &= \int \frac{f_Y(l)}{\EE Y_1}\,\1\{a \leq 0 \leq a+l\}\,da = 
    \int_{-l}^0 \frac{f_Y(l)}{\EE Y_1}\,da =  \frac{lf_Y(l)}{\EE Y_1}.
\end{align*}
Let $f_{A|L=l}(a)$ be the conditional probability density function for the left-most
point $A$ of an interval given its length $L$. Using the definition of conditional density,
\begin{align*}
    f_{A\,|L=l}(a) = \frac{f(a,l)}{f(l)} = \frac{
    \frac{ f_Y(l) }{ \EE Y_1 } \, \1\{a \leq 0 \leq a+l \} }{
    \frac{l f_Y(l)}{\EE Y_1 } }
      = \frac{1}{l}\,\1\{a \in [-l,0]\}.
\end{align*}
Thus $A \sim \text{Unif}[-l, 0]$.
\end{proof}

When the censoring probability does not depend on $\zeta$, the latter may be 
estimated by treating the non-degenerate intervals as an independent sample 
from $f(l)$ and applying the maximum likelihood method. For example, if $f_Y$ 
is the probability density function of a $\text{Gamma}(k, \lambda)$ distribution with shape 
parameter $k > 0$ and rate parameter $\lambda > 0$, $f(l)$ is the probability 
density of a Gamma distribution with parameters $k+1$ and $\lambda$.

\subsection{State estimation}
\label{sec:state}

Since the posterior distribution of $X$ or $W$ given $U$ (cf.\
Theorem~\ref{prop:cond_prob_viable}) is intractable because of the 
normalisation constant $c(\uu)$, we will use Markov chain Monte Carlo methods 
\citep{Brooetal11,MollWaag03} for simulation. These methods construct a Markov 
chain in such a way that the stationary distribution of the chain is exactly the posterior 
distribution. Of these methods, a Metropolis-Hastings algorithm with a fixed 
number of points will be used. Since the transition probabilities depend on 
likelihood ratios, the benefit is that one can sample from unnormalised densities.

Let us return to the framework of Theorem~\ref{prop:cond_prob_viable}. Note that 
sampling from the posterior distribution of $X$ given $U$ is cumbersome due to the 
presence of the permutation sum term in (\ref{eq:jn}).  Therefore our approach is 
to sample from the posterior distribution of the complete model $W$ and project on 
its ground process of locations. Doing so avoids attributing points to intervals and 
therefore avoids the intractable sum. Moreover, as we saw in 
Section~\ref{sec:model}, $W$ has probability density function $p_X$ with respect 
to a unit rate Poisson process on $\cX \times (\R\times\R^+)$ with intensity
measure $\ell \times \nu$.  Upon observing $U = \uu$ for
$\uu = \{ (a_1, 0), \dots, (a_m, 0), (a_{m+1}, l_{m+1}), \dots, (a_n, l_n) \}$, by
 (\ref{eq:jn}) this means that we must sample from a probability density function $\pi$ 
 on $\cX^{n-m}$ that is proportional to $p_X( \{ a_1, \dots, a_m, x_1, \dots, x_{n-m} \})$.
The ordering of the points inherent in working on $\cX^{n-m}$ represents
the unique correspondence between points in $X$ and intervals in $U$ in the 
complete model. We will use the notation $\bar \xx$ to indicate that we look at 
vectors rather than sets $\xx$. In the special case that $n=m$, all points are 
observed perfectly and there is no need for any simulation.

We will use the Metropolis--Hastings algorithm \citep{Brooetal11} when $n>m$, i.e.\
when there are density-admitting points.
The state space is given by
\begin{align*}
    \overline E(\uu) = \{ (x_1, \dots, x_{n-m}) \in \cX^{n-m}: 
    x_i \in \cX \cap [a_{m+i}, a_{m+i} + l_{m+i}], \,
    p_X( \{ a_1, \dots, a_m, x_1, \dots, x_{n-m} \} ) > 0 \}.
\end{align*}
From now on, we shall assume that the state space is non-degenerate in the
sense that
\begin{equation}
\label{eq:data_viable}
\int_{\overline E(\uu)} p_X( \{ a_1, \dots, a_m, x_1, \dots, x_{n-m} \} ) \, 
dx_1 \dots dx_{n-m} > 0.
\end{equation}

Now, the Metropolis--Hastings algorithm is defined as follows.
Let $q: \overline E(\uu) \times \overline E(\uu) \to \R^+$ be a Markov kernel.
Iteratively, if the current state is $\bar \xx \in \overline E(\uu)$,
propose a new state $\bar \yy \in \overline E(\uu)$ according to the probability 
density function $q(\bar \xx, \cdot)$ and accept the proposal to move to $\bar \yy$ with probability
\begin{equation}
\alpha( \bar \xx, \bar \yy ) = \left\{ \begin{array}{ll}
1 &  \mbox{ if } p_X( \{ a_1, \dots, a_m, y_1, \dots, y_{n-m} \} )\, q(\bar \yy, \bar \xx) \\
  & \quad \quad \geq p_X( \{ a_1, \dots, a_m, x_1, \dots, x_{n-m} \} )\,q(\bar \xx, \bar \yy);\\
 \frac{ p_X( \{ a_1, \dots, a_m, y_1, \dots, y_{n-m} \} ) \, q(\bar \yy, \bar \xx) }{ 
        p_X( \{ a_1, \dots, a_m, x_1, \dots, x_{n-m} \} ) \, q(\bar \xx, \bar \yy)\}}
 & \mbox{ otherwise.}
\end{array}
\right.
\label{eq:MH-general}
\end{equation}
When the proposal is rejected, stay in the current state $\bar \xx$. The choice
of $q$ depends on $p_X$. In our simulations in Section~\ref{sec:sim}, we will use the 
following algorithm which is valid when the prior density function $p_X$ is strictly positive. 

\begin{algorithm}
\label{algo:MH}
Supppose that $p_X > 0$ and $n>m$.  Iteratively, if the current state is 
$\bar \xx \in \overline E(\uu)$,
\begin{itemize}
\item pick an interval $[a_{m+i}, a_{m+i} + l_{m+i}]$, $i=1, \dots, n-m$, uniformly
at random from the non-degenerate ones;
\item generate a uniformly randomly distributed point $y_i$ on 
$\cX \cap [a_{m+i}, a_{m+i} + l_{m+i}]$ and propose to update $x_i$ to $y_i$;
\item accept the proposal with probability
\begin{equation}
\alpha_i( (x_1, \dots, x_{n-m}), y_i ) = \min \left( 1, 
 \frac{ p_X( (\{ a_1, \dots, a_m, x_1, \dots, x_{n-m} \} 
   \setminus \{ x_i \} ) \cup \{ y_i \}) 
}{ p_X( \{ a_1, \dots, a_m, x_1, \dots, x_{n-m} \}) }
\right)
    \label{eq:hastings_ratio}
\end{equation}
and otherwise stay in the current state.
\end{itemize}
\end{algorithm}

A few remarks are in order. First, note that since $\cX$ is open, the intersection with 
closed intervals that contain a point in $\cX$ is also non-degenerate when $l_i > 0$.  
Secondly, when $p_X$ may take the value zero, the proposal mechanism in Algorithm~\ref{algo:MH} 
might result in a new state that does not belong to $\overline E(\uu)$, even when $\bar \xx$ does. 
Moreover, only changing one component at a time might lead to non-irreducible Markov chains. 
For example, if $\uu$ contains the parametrisations of the intervals $[0,1]$ and $[ 0.1, 1]$ 
and $p_X(\xx) = 0$ for realisations $\xx$ that contain components separated by a distance 
less than $0.55$, then states such as $\bar \xx = (0.3, 0.9)$ and $\bar \yy = (0.9, 0.3)$ 
cannot be reached from one another.

Let the target distribution be $\pi$ given by (\ref{eq:jn}) interpreted as a probability
density on $\overline E(\uu)$.
In the next propositions, basic properties of the algorithm are considered. The proofs 
are modifications to our context of classic Metropolis-Hastings proofs found in, for example, \citep{MengTwee96},
\citep{RobeSmit94} or \citep[Chapter~7]{MollWaag03}.

We will write $Y_i$ for subsequent states and denote by $P(\bar \xx, F) = 
P(Y_{i+1} \in F\,|\,Y_i = \bar \xx)$ the transition probability
from state $\bar \xx \in \overline E(\uu)$ into $F \subset \overline E(\uu)$. 

\begin{proposition}
Consider the set-up of Theorem~\ref{prop:cond_prob_viable} with $n>m$ and assume that 
condition (\ref{eq:data_viable}) is met. Then, the Metropolis-Hastings algorithm 
defined by Markov kernel $q$ on $\overline E(\uu)$ and acceptance probabilities 
(\ref{eq:MH-general}) is reversible with respect to $\pi$.

\label{prop:rev}
\end{proposition}

\begin{proof}[Proof.] 
Take $\bar \xx, \bar \yy$ in $\overline E(\uu)$ and assume that
$ \pi(\bar \yy) \, q(\bar \yy, \bar \xx)$ $ > \pi(\xx)\,  q(\bar \xx, \bar \yy) \geq 0$.
Then
\begin{align*}
  \pi(\bar \xx) \, q(\bar \xx, \bar \yy) \, \alpha(\bar\xx, \bar \yy) = 
  c(\uu) \, p_X(\{ a_1, \dots, a_m, x_1, \dots, x_{n-m} \}) \, q(\bar \xx, \bar \yy) =
\end{align*}
\begin{align*}
  c(\uu) \, p_X( \{ a_1, \dots, a_m, y_1, \dots, y_{n-m} \} ) \, q(\bar \yy, \bar \xx) \,
   \frac{ p_X( \{ a_1, \dots, a_m, x_1, \dots, x_{n-m} \} )  \, q(\bar \xx, \bar \yy)}
        { p_X( \{ a_1, \dots, a_m, y_1, \dots, y_{n-m} \} )  \, q(\bar \yy, \bar \xx)} 
   =  \pi(\yy) \, q(\bar \yy, \xx) \, \alpha(\bar\yy, \xx)
\end{align*}
writing $c(\uu)$ for the normalisation constant.
We conclude that the chain is in detailed balance and therefore reversible with 
respect to $\pi$.
\end{proof} 

Recall that the Markov chain is called $\pi$-irreducible \citep{MeynTwee09} 
if for every $\bar \xx \in \overline E(\uu)$ and every $F \subset \overline E(\uu)$ with $\pi(F) > 0$ 
there exists some natural number $\tau$ such that $P^{\tau}(\bar \xx, F) > 0$.

\begin{proposition}
Consider the set-up of Theorem~\ref{prop:cond_prob_viable} with $n>m$ and assume that 
condition (\ref{eq:data_viable}) is met. Let $Q$ be the one-step transition kernel of the Markov
chain on $\overline E(\uu)$ generated by Markov kernel $q: \overline E(\uu) \times
\overline E(\uu) \to \R^+$ in which every 
proposal is accepted. If the chain defined by $Q$ is $\pi$-irreducible and 
$q(\bar \xx, \bar \yy) = 0$ if and only if $q(\bar \yy, \bar \xx) = 0$, then the 
Metropolis-Hastings algorithm defined by $q$ and (\ref{eq:MH-general}) is $\pi$-irreducible. 
In particular, the chain of Algorithm~\ref{algo:MH} is $\pi$-irreducible when $p_X > 0$.
\label{prop:irr}
\end{proposition}

\begin{proof}[Proof.] 
The first part follows from \citep[Theorem~3.ii]{RobeSmit94}. 

For Algorithm~\ref{algo:MH}, $q(\bar \xx, \bar \yy) > 0$ only if
$\bar \xx$ and $\bar \yy$ in 
$\overline E(\uu)$ differ in at most a single component. Thus, assume that $x_j = y_j$ 
for all $j\neq i \in \{ 1, \dots, n-m \}$ and $x_i \neq y_i$. Then $q(\bar \xx, \bar \yy) 
= q(\bar \yy, \bar \xx)$ so they are strictly positive or zero together. Write $Q^{\tau}$ 
for the $\tau$-step transition kernel of the aways-accept chain. Then, for $\bar \xx$, 
$\bar \yy \in \overline E(\uu)$,
\[
q^{n-m}(\bar \xx, \bar \yy) \geq 
\left(\frac{1}{n-m}\right)^{n-m} \,\prod_{i=1}^{n-m} \frac{1}{\ell( \cX \cap
[a_{m+i}, a_{m+i} + l_{m+i} ])} > 0
\]
by changing each component in turn. We conclude that the Markov chain of
Algorithm~\ref{algo:MH} is $\pi$-irreducible.
\end{proof}

Recall that a $\pi$-irreducible Markov chain is called aperiodic \citep{MeynTwee09}
if the state space ($\overline E(\uu)$ in our case) cannot be partitioned into 
measurable sets $B_0$, $B_1, \dots B_{d-1}$ such that 
$\pi(\overline E(\uu) \setminus \cup_{j=0}^{d-1} B_j) = 0$ and 
$P(\bar \xx, B_{j+1\,\text{mod}\,d })$ $ = 1$ for all $\bar \xx \in B_j$ (for some $d>1$,
the period). By \citep[Proposition~7.6]{MollWaag03}, a $\pi$-irreducible Markov chain
with invariant probability distribution $\pi$ is aperiodic if and only if
for some small set $D$ with $\pi(D)>0$ and some $\tau\in\N$, the following holds:
\(
P^i(\bar \xx, D) > 0
\)
for all $\bar \xx\in D$ and $i\geq \tau$.

\begin{proposition}
Consider the set-up of Theorem~\ref{prop:cond_prob_viable} with $n>m$ and assume that 
condition (\ref{eq:data_viable}) is met. If $0 < p_X(\{ a_1, \dots, $ $a_m,$  
$x_1, \dots, x_{n-m} ) \leq \delta$ for some $\delta > 0$ and all $\bar \xx
\in \overline E(\uu)$, then the Metropolis-Hastings Markov chain of 
Algorithm~\ref{algo:MH} is aperiodic.
\label{prop:aperiodic}
\end{proposition}

\begin{proof}[Proof.]
Let $\xi$ be the point on $\cX \cap [a_{m+1}, a_{m+1} + l_{m+1}]$ that replaces $x_1$. By (\ref{eq:data_viable}), there exist $x_2, \dots, x_{n-m}$ such that 
\[
 \int_{ \cX \cap [ a_{m+1},  a_{m+1}+l_{m+1}] } 
c(\uu) \,
 p_X( \{ a_1, \dots, a_m, \xi, x_2, \dots, x_{n-m}) \, d\xi
 = \int_{ \cX \cap [ a_{m+1},  a_{m+1}+l_{m+1}] } 
  \pi(\xi, x_2, \dots, x_{n-m} ) \, d \xi 
\]
is strictly positive.
Define a measure $\mu$ on the Borel $\sigma$-algebra on $\overline E(\uu)$ by
\[
\mu(F) =  \int_{ \cX \cap [ a_{m+1},  a_{m+1}+l_{m+1}] } 
\1\{ (\xi, x_2, \dots, x_{n-m}) \in F \} \,
  \pi(\xi, x_2, \dots, x_{n-m} ) \, d\xi
\]
and note that $\mu(\overline E(\uu) ) > 0$.
Set $C = \{ ( \xi, x_2, \dots, x_{n-m}) :  \xi \in \cX \cap [a_{m+1}, a_{m+1} + l_{m+1}] \}$.
We claim that $C$ is small with respect to $\mu$. To see this, 
take $\bar \yy = (y, x_2, \dots, x_{n-m}) \in C$ and note that for $F \subset \overline E(\uu)$, the transition probability
\(
P(\bar \yy, F) 
\)
is at least
\[
\frac{1}{n-m} \frac{1}{\ell(\cX \cap [a_{m+1},a_{m+1}+l_{m+1}] )} 
\int_{ \cX \cap [ a_{m+1},  a_{m+1}+l_{m+1}] } 
\1_F(\xi, x_2, \dots, x_{n-m}) \,
\alpha_1( (y, x_2, \dots, x_{n-m}), \xi) \, d\xi.
\]
If $\pi(\xi, x_2, \dots, x_{n-m}) > \pi(y, x_2, \dots, x_{n-m})$,
then $\alpha_1( (y, x_2, \dots, x_{n-m}), \xi) = 1 > 
\pi(\xi, x_2, \dots, x_{n-m}) / (c(\uu) \, \delta)$, again
writing $c(\uu)$ for the normalisation constant. Otherwise, 
$\alpha_1( (x_1, \dots, x_{n-m}), \xi) = 
\pi(\xi, x_2, \dots, x_{n-m}) / $ 
$\pi(x_1, x_2, \dots, x_{n-m}) 
>  \pi(\xi, x_2, \dots, x_{n-m}) / (c(\uu) \, \delta )$. In summary, 
\[
P(\bar \xx, F) > 
\frac{1}{n-m} \frac{1}{\ell(\cX \cap [a_{m+1},a_{m+1}+l_{m+1}])} \, \frac{\mu(F)}
{c(\uu) \, \delta}
\]
so $C$ is small with respect to $\mu$.  Moreover, $\pi(C) = \mu(\overline E(\uu)) > 0$.
Iterating the above argument one notices that  $P^\tau(\bar \xx,F)$ is at least
as large as the $\tau$-th power of the bound above, an observation that completes
the proof.
\end{proof}

In conclusion, from almost all initial states, Algorithm~\ref{algo:MH} converges
in total variation to the invariant probability distribution. Conditions for 
general proposal kernels $q$ can be found in \citep[Chapter~7]{MollWaag03} or
\citep[Theorem~3]{RobeSmit94}.

\subsection{Prior model parameters}
\label{sec:EMalgo}

In Section~\ref{sec:state}, we discussed Monte Carlo methods to sample from the
posterior distribution of $W$ or $X$ given $U=\uu$. This distribution is defined
in terms of  the prior probability density function $p_X$. 
Typically, $p_X$ is given in unnormalised form and depends on a parameter 
vector $\theta$, that is,
\(
p_X(\xx; \theta) = c(\theta) h_X( \xx; \theta)
\)
for an explicit function $h_X: \cN_\cX \to \R^+$. When $\theta$ is treated as
an unknown, since the likelihood 
function for $\theta$ contains the latent marked point process $W$, we call on techniques
from missing data analysis.


The likelihood function $l(\theta)$ is obtained from the proof of Theorem~\ref{prop:cond_prob_viable} 
by taking  $A$ equal to $\cN_\cX$ in equation (\ref{eq:final_rhs}). Disregarding
terms that do not depend on $\theta$, one obtains $l(\theta; \uu) = c(\theta) c(\theta|\uu)^{-1}$
where
\[
c(\theta)^{-1} = \sum_{k=0}^\infty \frac{e^{-\ell(\cX)}}{k!} 
\int_{\cX^k} h_X( \{ x_1, \dots, x_k \}; \theta ) \,
dx_1 \dots dx_k
\]
and
\(
c(\theta|\uu)^{-1} 
\)
is given by
\[
 \int_{\cX^{n-m}}
h_X( \{ a_1, \dots, a_m, x_1, \dots, x_{n-m} \}; \theta ) 
\left(
\sum_{ \substack{D_1, \dots, D_{n-m}   \\
       \cup_i \{ D_i \} = \{ 1, \dots, n-m \} } }
\1\{ x_{D_i} \in [a_{m+i}, a_{m+i}+l_{m+i}] \} 
\right) \prod_{i=1}^{n-m} dx_i.
\]
One observes that $c(\theta|\uu)$ is equal to the normalisation constant on
$\overline E(\uu)$ of the non-atomic part of the posterior distribution
of $X$ given $U=\uu$.  

To handle the two normalisation constants, it is necessary to look
at the log relative likelihood $L(\theta)$ of $U$ with respect to some fixed 
and user-selected reference parameter $\theta_0$,
Then, as in \citep{GelfCarl93,Geye99},
\[
L(\theta) = \log \left[ \frac{c(\theta) \, c(\theta_0|\uu) }{c(\theta_0) \, c(\theta | \uu)}
\right] =
\log \EE_{\theta_0} \left[
\frac{h_X(X; \theta)}{h_X(X; \theta_0)} \mid U =  \uu
\right] 
-
\log \EE_{\theta_0} \left[
\frac{h_X(X; \theta)}{h_X(X; \theta_0)}
\right] .
\]
Being expressible in terms of expectations under the reference parameter, the log 
likelihood ratio can be approximated by Markov chain Monte Carlo methods. Note that
two samples are required: one from the posterior distribution of $X$ and one from
the prior. For the latter, provided $p_X$ is locally stable, classic 
Metropolis--Hastings methods based on births and deaths apply \citep{Geye99,MollWaag03}. 
If the conditional intensity is monotone, exact simulation can be carried out 
\citep{KendMoll00,LiesBadd02}. 

\section{Examples}
\label{sec:sim}

In this section, we present a few examples to illustrate how the choice of prior affects state estimation. Calculations were carried out using the
{\tt C++} marked point process library MPPLIB, developed by Steenbeek et al. 
For $p_X$ we choose the area-interaction point process \citep{WidoRowl70, Badd95},
a model that favours clustered, regular and random realisations depending on parameter values. Specifically, this model has probability density function
\begin{align}
    p(\mathbf{x}) = \alpha \beta^{n(\mathbf{x})}\exp\left[-\log \gamma\, 
       \ell( \cX \cap U_r(\mathbf{x}) ) \right]
    \label{eq:area_int}
\end{align}
with respect to a unit rate Poisson process on $\mathcal{X}$.
Here $U_r(\mathbf{x}) = \bigcup_{i=1}^n B(x_i, r)$ where $B(x_i,r)$ is the closed
interval $[ x_i-r, x_i + r]$. When $\gamma < 1$, realisations tend to be regular, 
for $\gamma > 1$ clustered. When $\gamma = 1$, one has a Poisson process with 
intensity $\beta$. The scalar $\alpha$ is a normalisation constant. 
Realisations can be obtained by Kendall's dominated coupling from the past (CFTP) 
algorithm \citep{Kend98} developed initially from the perfect simulation methods 
of Propp and Wilson for coupled Markov chains \citep{ProppWilson96}. 

\subsection{Toy example}
Consider data $\uu = \{ (0.45, 0.4), (0.51, 0), (0.58, 0) \}$ that consist of two atoms
and a single non-degenerate interval. By the discussion at the end of 
Section~\ref{sec:posterior}, for a Poisson prior ($\gamma = 1$), the posterior distribution of the location $X_3$ in $\cX = (0,1)$ 
that generated the non-degenerate interval is uniformly distributed. 
To see the effect of informative priors, 
Figure~\ref{fig:clustered_structure} plots the posterior distribution of $X_3$
when the prior is an area-interaction model with $\eta = 2 r \log \gamma = 1.2$ and
$r=0.1$. Note that mass is shifted to the left side of the interval due to the 
presence of atoms. For $\eta = -1.2$ and $r=0.1$, the atoms repel $X_3$, 
resulting in mass being shifted to the right side of the interval
(cf.\ Figure ~\ref{fig:regular_structure}). To carry out 
the state estimation, we ran Algorithm~\ref{algo:MH} with a burn-in
of 10,000 steps and calculated the histograms based on the subsequent
100,000 steps.

\begin{figure}[H]
    \begin{subfigure}{.45\textwidth}
        \centering
        \includegraphics[width=\textwidth]{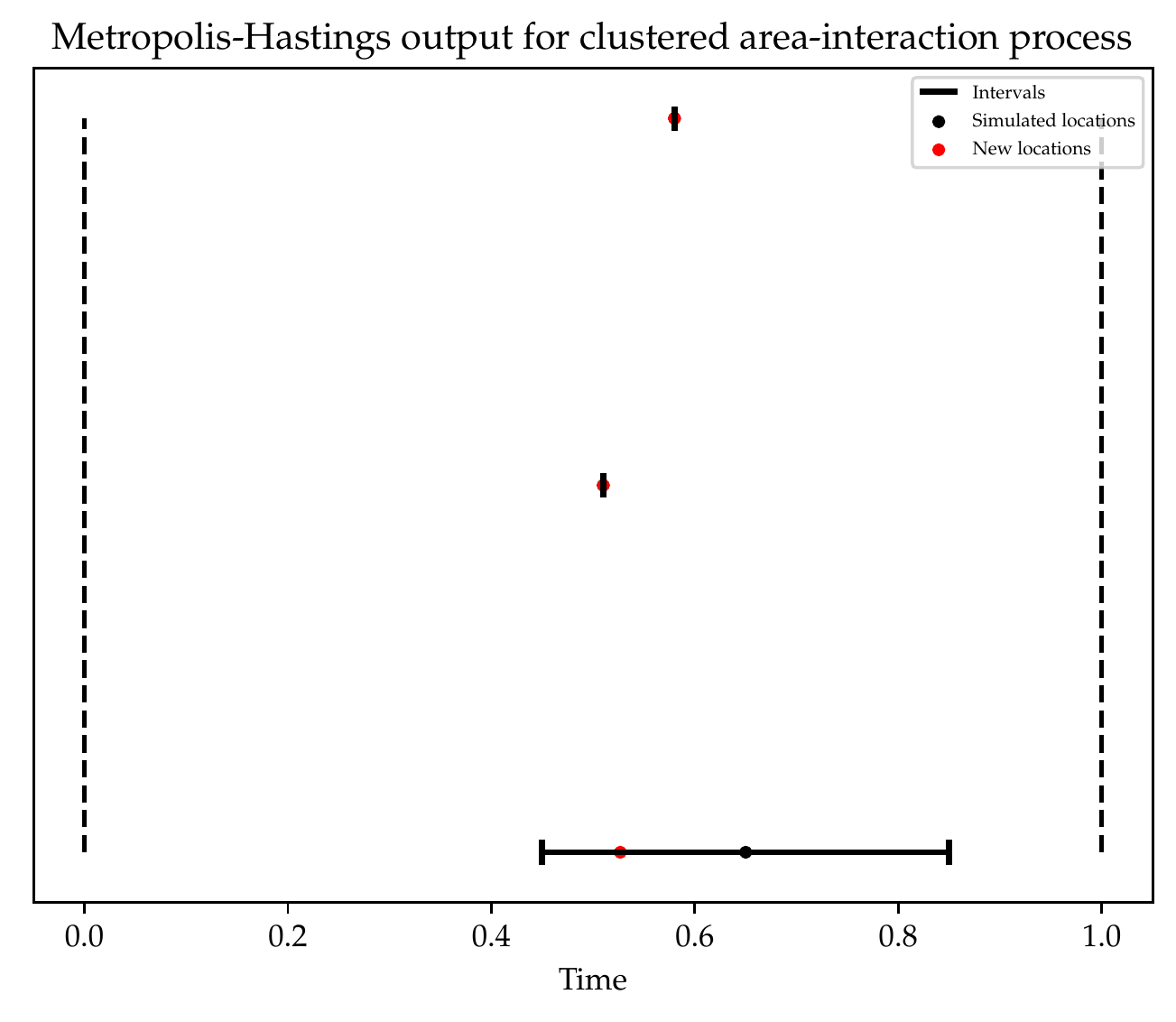}
    \end{subfigure}
    \begin{subfigure}{.45\textwidth}
        \centering
        \includegraphics[width=\textwidth]{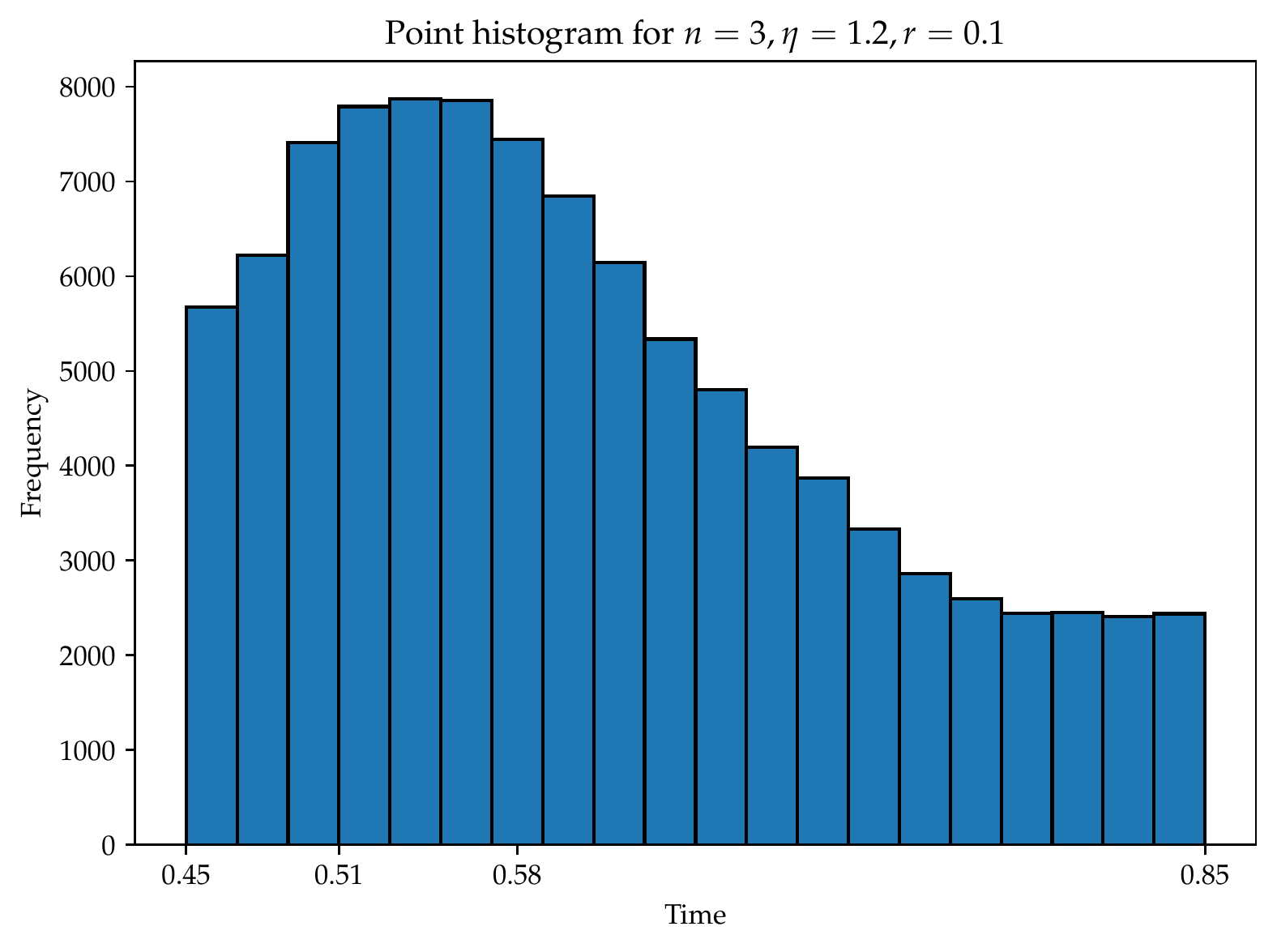}
    \end{subfigure}
        \caption{Locations of two atoms and a spanning interval together with a histogram of point locations. The interval start and end points as well as the atom locations are marked on the histogram $x$-axis.}
    \label{fig:clustered_structure}
\end{figure}

\begin{figure}[H]
    \begin{subfigure}{.45\textwidth}
        \centering
        \includegraphics[width=\textwidth]{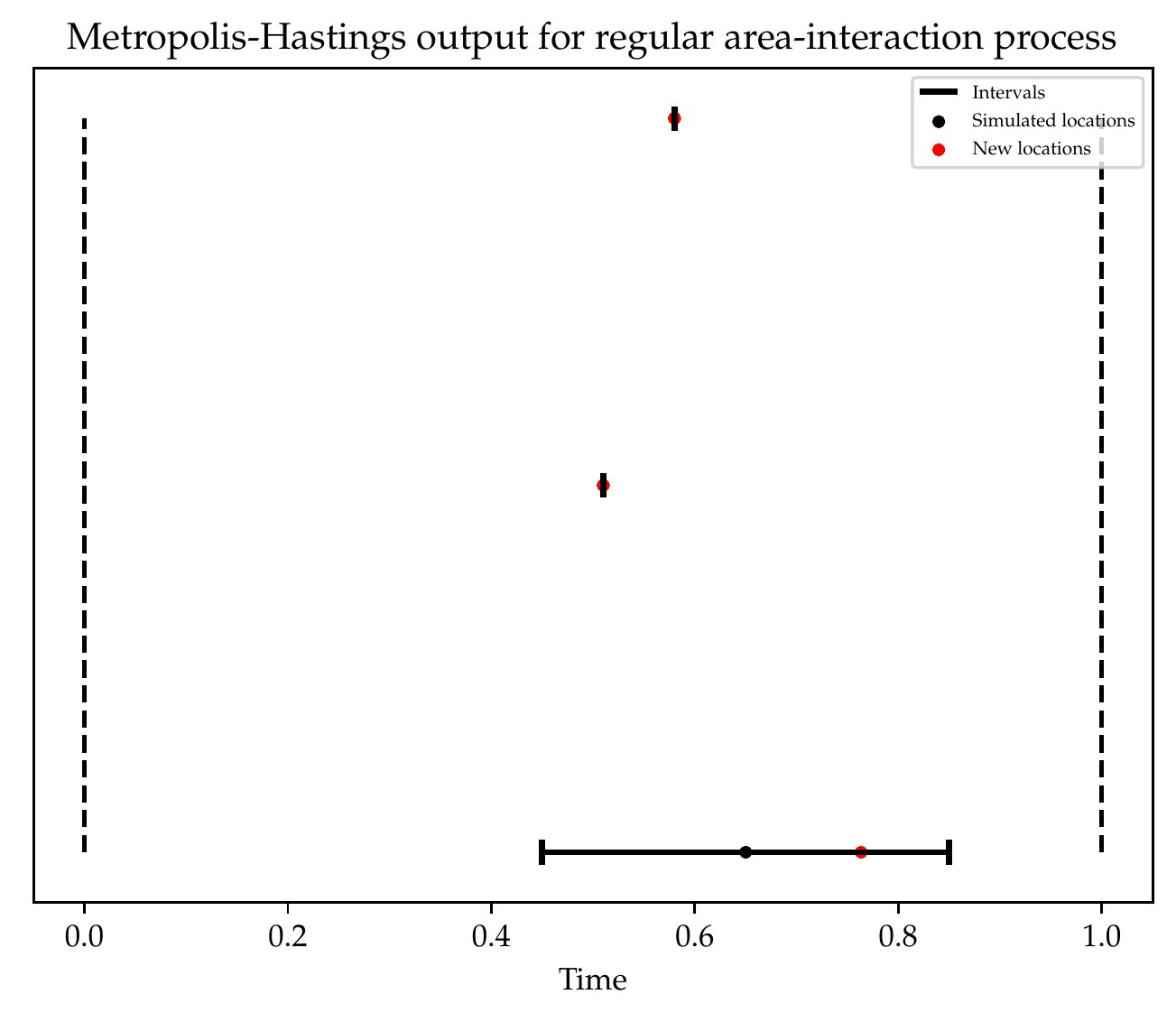}
    \end{subfigure}
    \begin{subfigure}{.45\textwidth}
        \centering
        \includegraphics[width=\textwidth]{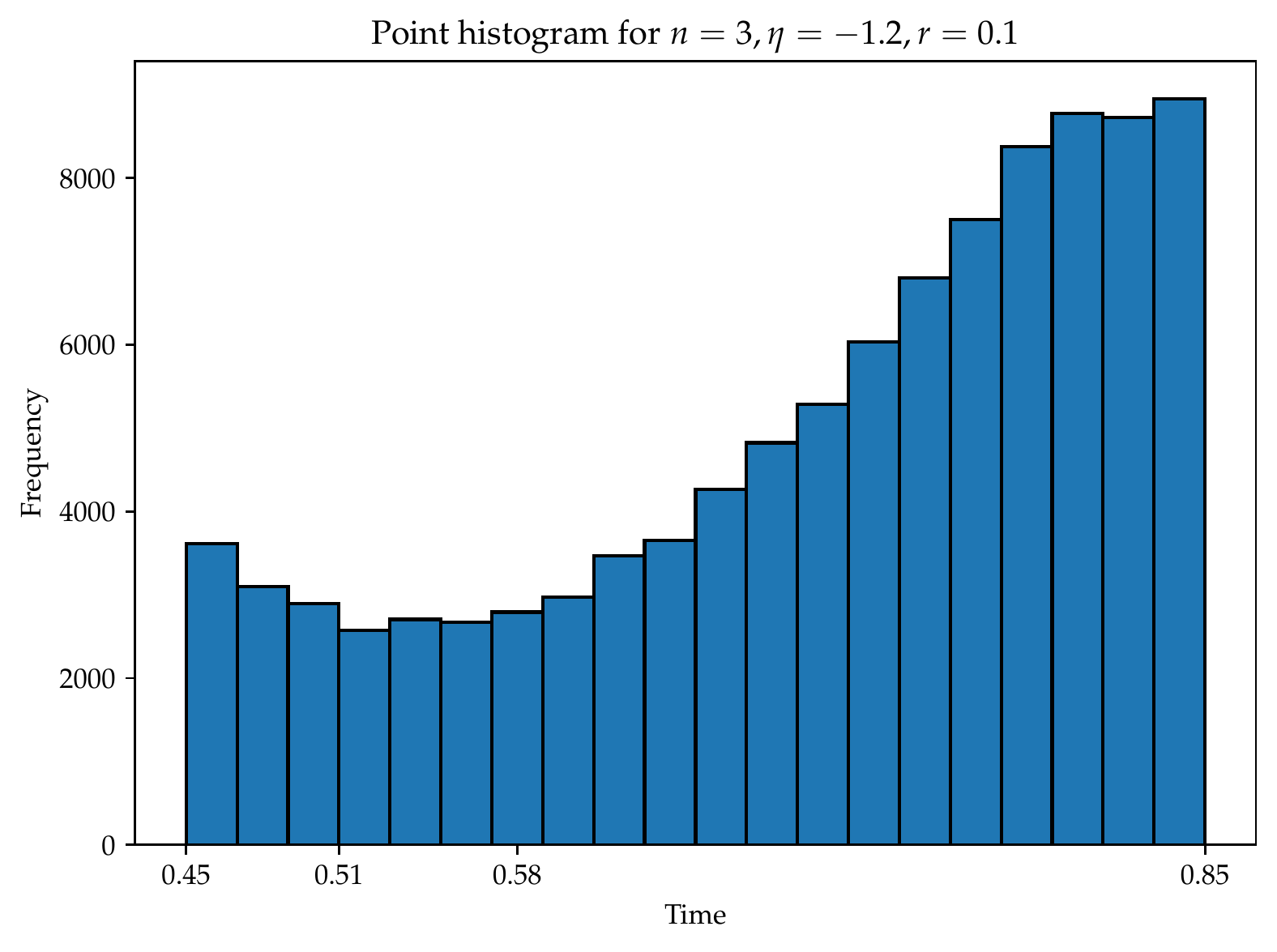}
    \end{subfigure}
        \caption{Locations of two atoms and a spanning interval together with a histogram of point locations. The interval start and end points as well as the atom locations are marked on the histogram $x$-axis.}
    \label{fig:regular_structure}
\end{figure}

\subsection{Area-interaction gamma model}

The left-most panel in Figure~\ref{fig:random_plot} shows a simulation in $\cX = (0,1)$ from 
$U$ in a model where $X$ is an area-interaction process with parameters 
$\beta = 12$, $\eta = 2 r \log \gamma = 0$ and $r=0.05$ marked by a mixture distribution 
$\nu$ in which the atom probability is $p=0.2$ and $f_Y$ is the probability density 
function of a Gamma distribution with shape parameter $k=2.5$ and rate parameter 
$\lambda=0.07$. 
The points shown as black dots are the points of $X$ in the simulated pattern, the red
points constitute a realisation from the posterior distribution of $X$ given $U$ 
obtained by running Algorithm~\ref{algo:MH} for 10,000 steps. The points seem 
to settle in a random manner within the intervals.

\begin{figure}[H]
    \begin{subfigure}{.45\textwidth}
        \centering
        \includegraphics[width=\textwidth]{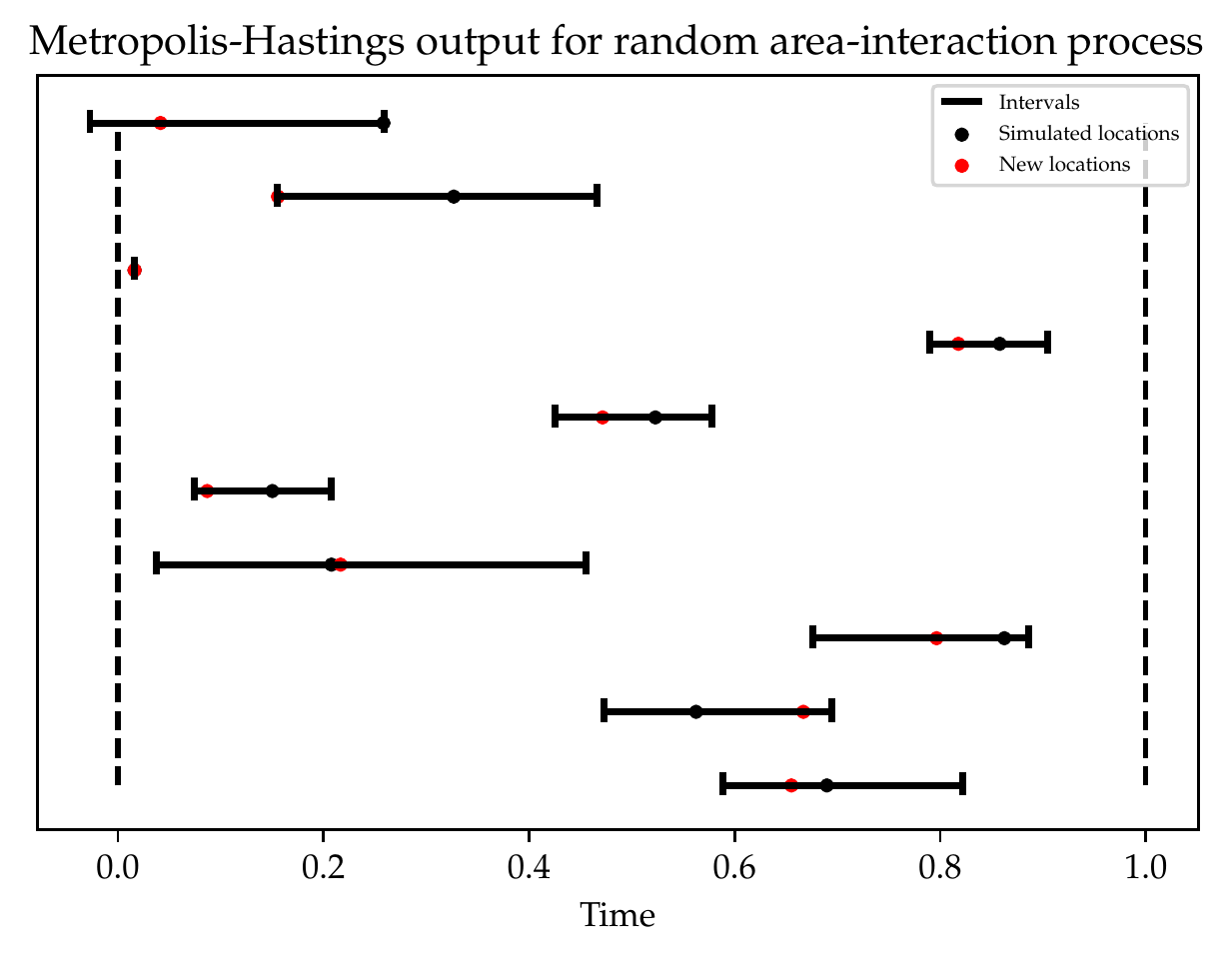}
    \end{subfigure}
    \begin{subfigure}{.45\textwidth}
        \centering
        \includegraphics[width=\textwidth]{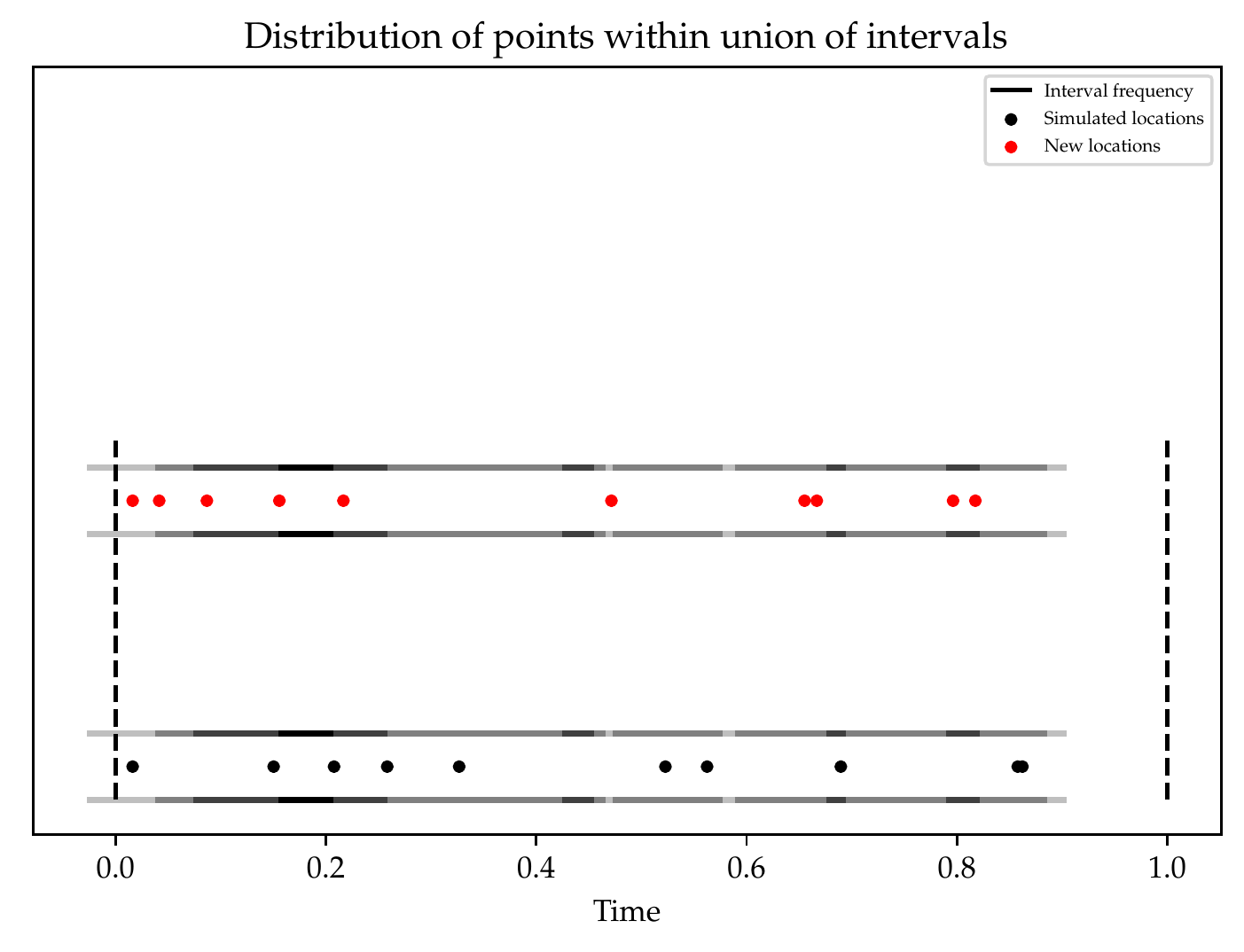}
    \end{subfigure}
        \caption{Plots of the simulated and new locations of a random area-interaction point process. Parameter values: $(\beta, \eta, r, \lambda, k, p) = (12, 0, 0.05, 0.07, 2.5, 0.2)$.}
    \label{fig:random_plot}
\end{figure}

A simulation using a prior favouring clustering can be found in black in 
Figure~\ref{fig:clustered_plot}. The parameter settings were as before except that 
$\eta = 1.2$. In red, a realisation from the posterior distribution is shown, obtained after 10,000
steps from the Metropolis--Hastings algorithm.  Figure~\ref{fig:clustered_plot} shows the effect of the complex underlying geometry when choosing proposal 
points within parametrised intervals. 
The algorithm tends to move proposed times to areas where multiple intervals intersect, 
leading to clustering within these regions. 

\begin{figure}[H]
    \begin{subfigure}{.45\textwidth}
        \centering
        \includegraphics[width=\textwidth]{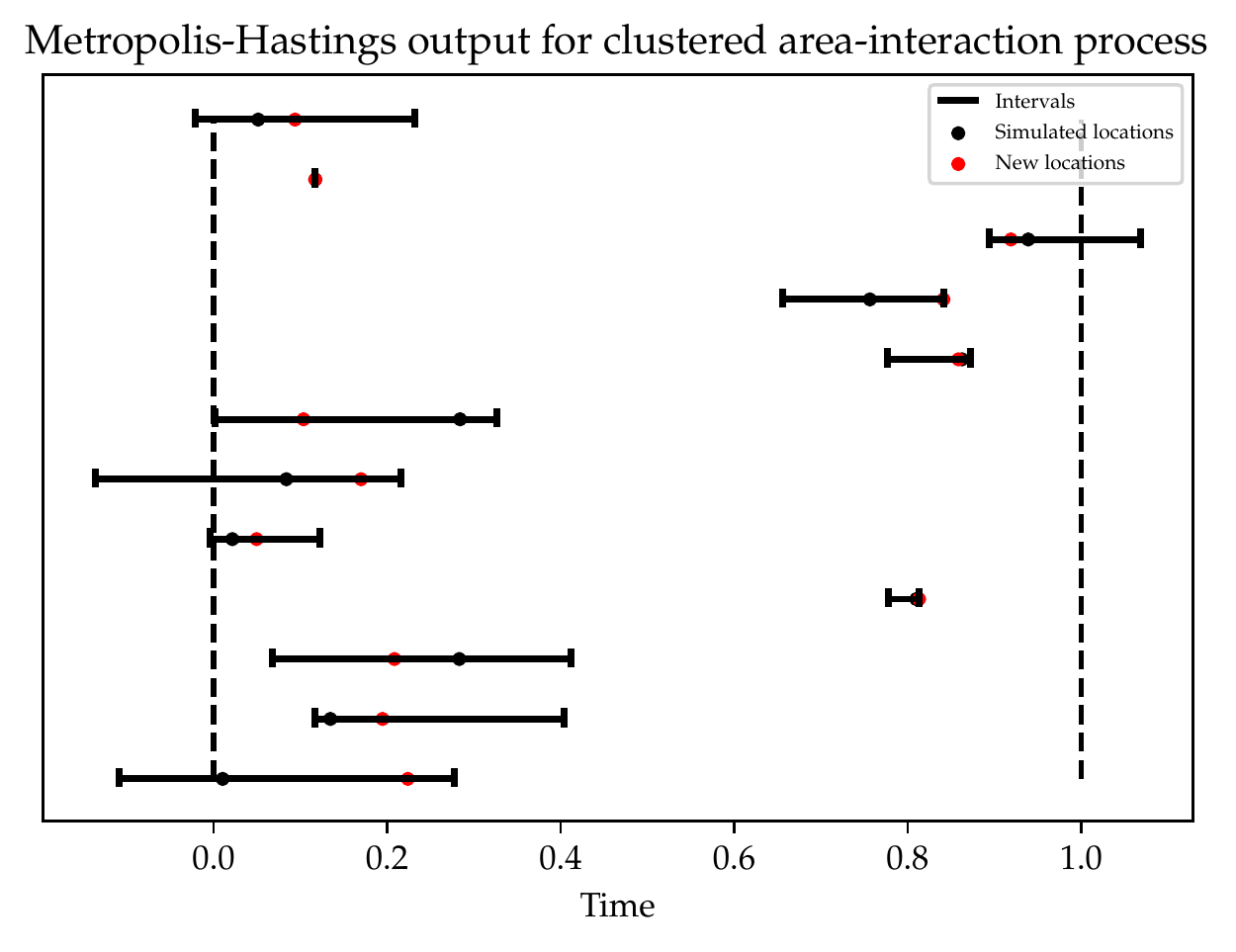}
    \end{subfigure}
    \begin{subfigure}{.45\textwidth}
        \centering
        \includegraphics[width=\textwidth]{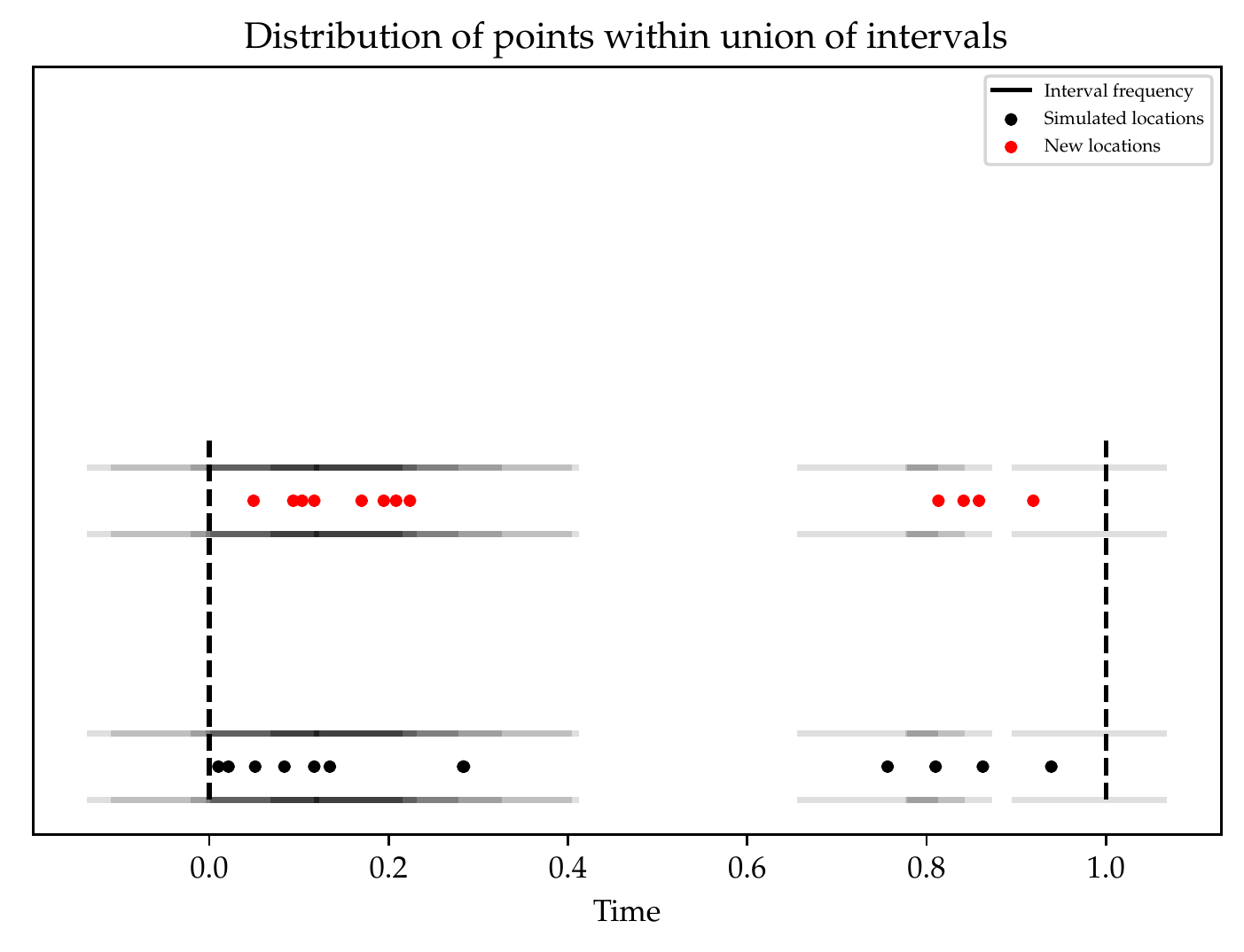}
    \end{subfigure}
        \caption{Parameter values: $(\beta, \eta, r, \lambda, k, p) = (12, 1.2, 0.05, 0.07, 2.5, 0.2)$.}
    \label{fig:clustered_plot}
\end{figure}

Figure~\ref{fig:regular_plot} shows, in black, a simulation from $U$ for a regular 
area-interaction prior with $\eta = 2 r \log \gamma = -1.2$. A realisation from the posterior 
distribution of $X$ is shown in red. The structure of the prior point process is 
maintained in the posterior, with points being spread out from each other. 

\begin{figure}[H]
    \begin{subfigure}{.45\textwidth}
        \centering
        \includegraphics[width=\textwidth]{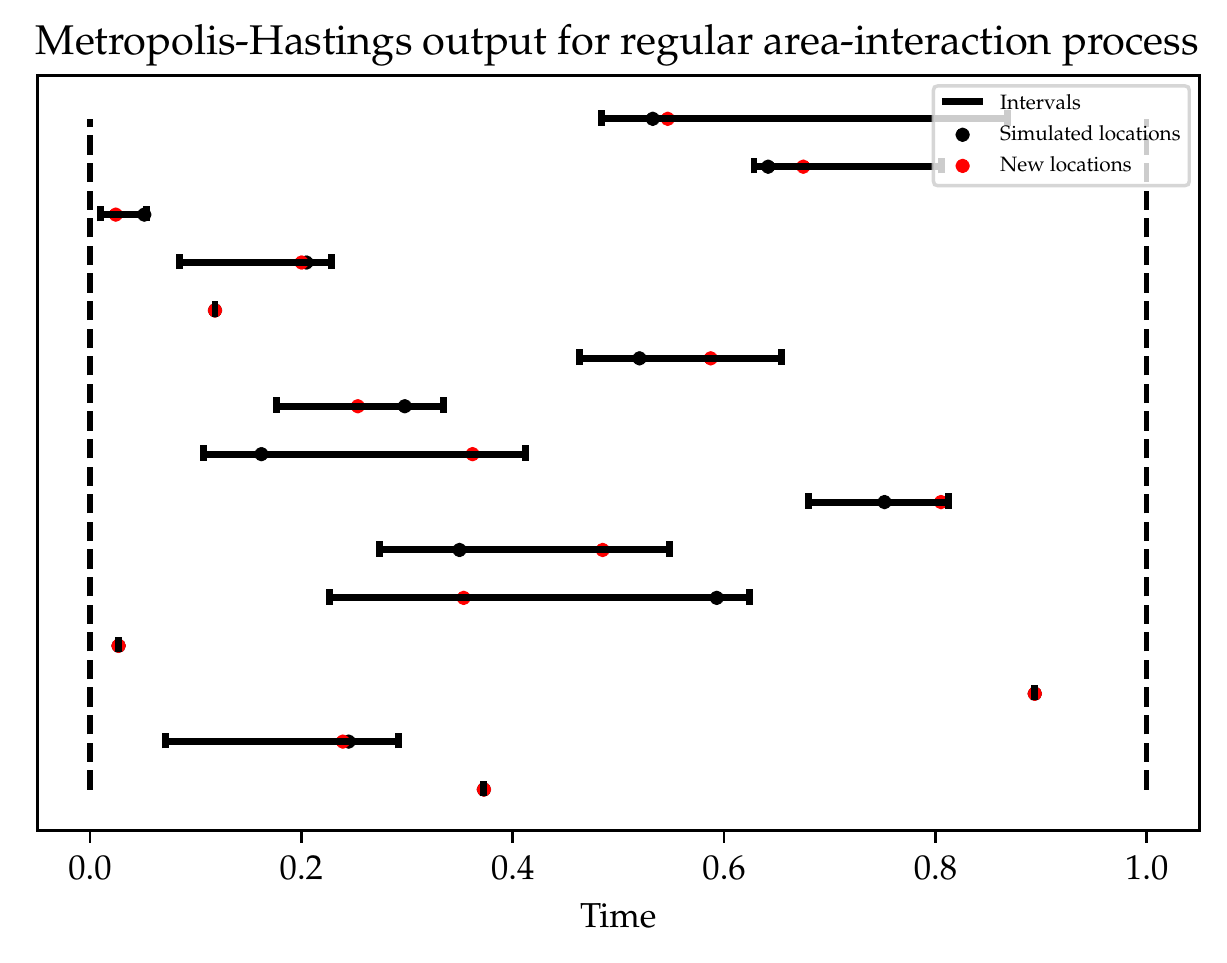}
    \end{subfigure}
    \begin{subfigure}{.45\textwidth}
        \centering
        \includegraphics[width=\textwidth]{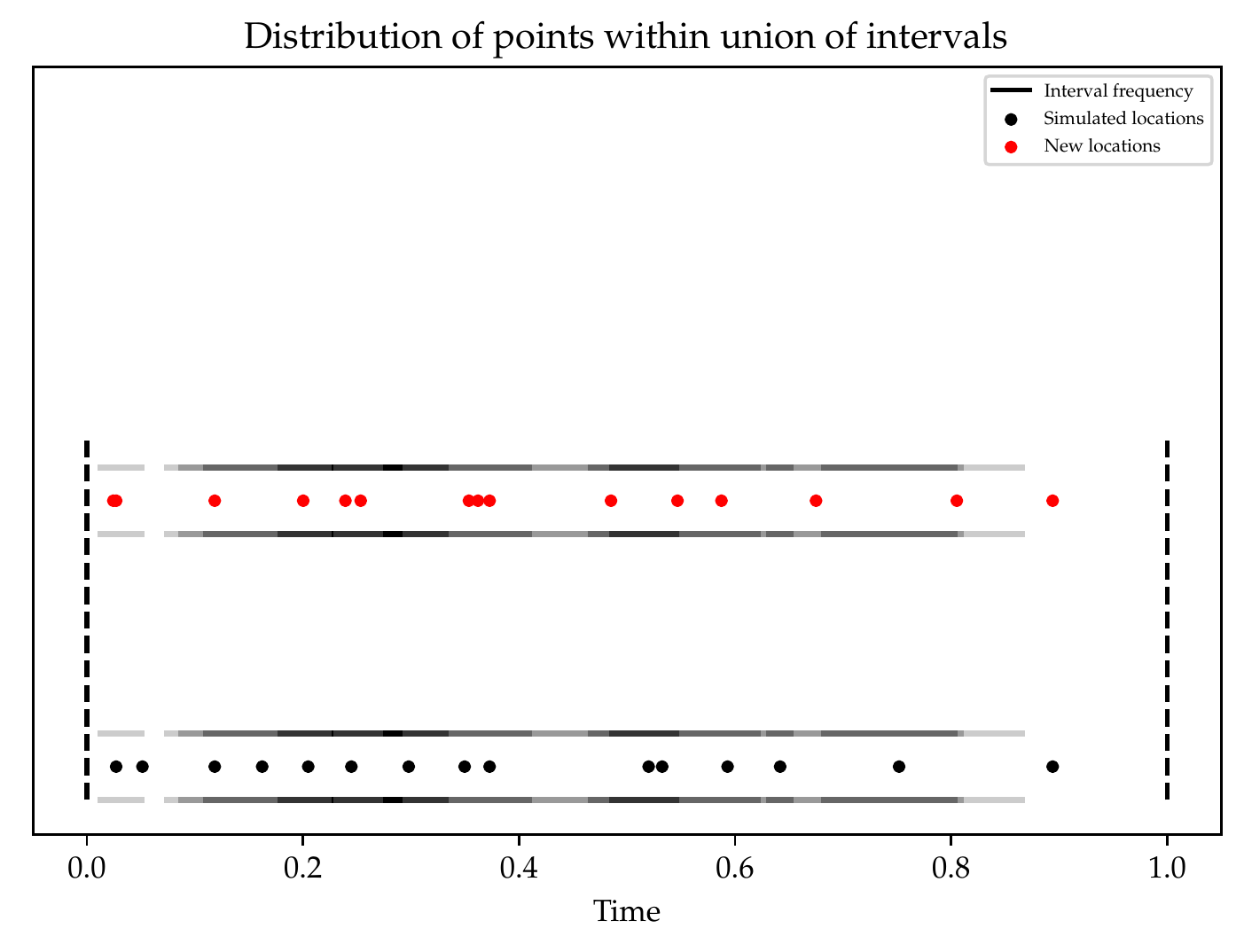}
    \end{subfigure}
        \caption{Parameter values: $(\beta, \eta, r, \lambda, k, p) = (12, -1.2, 0.05, 0.07, 2.5, 0.2)$.}
    \label{fig:regular_plot}
\end{figure}

\section{Conclusion}

In this work, a Bayesian inference framework for aoristic data was
introduced in which an alternating renewal process is used to interval
censor temporal data, converting it into a marked point process model. 
A prospective point, which cannot be observed directly, was paired with 
an interval within which the point surely lies. State estimation was
then applied to best estimate the location of this point. Theory was
developed regarding the distribution of these marks based on this renewal
framework and the posterior distribution  deduced. The fact that the
forward model allows for a mixture of discrete and absolutely continuous
components makes this process nontrivial. A state estimation procedure 
was outlined in the form of a Metropolis-Hastings algorithm for a fixed
number of points, after which ergodicity properties were verified. Using 
an area-interaction prior, this procedure was applied to sample from the
posterior distribution. Effects of the prior are clearly present when
sampling from the complete model.

Throughout, we assumed that all intervals corresponding to a point in
$\cX$ were observed. Returning to a criminology context, sampling bias 
may arise since the data may contain only intervals whose right-most 
point is in a given interval. Additionally, a random labelling regime was
assumed. It might be more realistic to have location-dependent independent
marking, for example based on a semi-Markov process rather than an
alternating renewal process. Furthermore, spatial aspects were completely
ignored. These generalisations will form the topic for our future research.

\section*{Acknowledgements}

This research was funded by NWO, the Dutch Research Council (grant OCENW.KLEIN.068).


\begin{thebibliography}{99}

\bibitem[Asmussen(2003)]{Asmu03}
Asmussen, S. (2003).
{\em Applied probability and queues\/}.
Springer.
\bibitem[Baddeley \& Van Lieshout(1995)]{Badd95}
Baddeley, A. \& Lieshout, M.N.M.~van. (1995).
Area-interaction point processes.
{\em Annals of the Institute of Statistical Mathematics, 47,} 601--619.
\bibitem[Bernasco(2009)]{Bern09}
Bernasco, W. (2009).
Burglary. In Tonry, M. (Ed.), {\em Oxford handbook of crime and
public policy}.
Oxford University Press (pp. 165--190).
\bibitem[Br\'emaud(1972)]{Brem72}
Br\'emaud, P. (1972).
A martingale approach to point processes (PhD thesis).
University of Berkeley.
\bibitem[Brix \& Diggle(2001)]{BrixDigg01}
Brix, A. \& Diggle, P.J. (2001).
Spatio-temporal prediction for log-Gaussian Cox processes.
{\em Journal of the Royal Statistical Society: Series B, \it 63}, 823--841.
\bibitem[Brooks, Gelman, Jones \& Meng(2011)]{Brooetal11}
Brooks, S., Gelman, A., Jones, G. \& Meng, X.L. (2011).
{\em Handbook of Markov chain Monte Carlo.} CRC Press.
\bibitem[Daley \& Vere-Jones(2003, 2008)]{DaleVere88}
Daley, D.J. \& Vere--Jones, D. (2003, 2008).
{\em An introduction to the theory of point processes\/} (2nd ed., in two
volumes). Springer.
\bibitem[Dempster, Laird \& Rubin(1997)]{Dempetal77}
Dempster, A.P., Laird, N.M. \& Rubin, D.B. (1977).
Maximum likelihood from incomplete data via the EM algorithm.
{\em Journal of the Royal Statistical Society: Series B, \it 39}, 1--38.
\bibitem[Gamerman \& Lopes(2006)]{Game06}
Gamerman, D. \& Lopes, H.F. (2006).
{\em Markov chain Monte Carlo: stochastic simulation for Bayesian inference\/}.
CRC Press.
\bibitem[Gelfand \& Carlin(1993)]{GelfCarl93}
Gelfand, A.E. \& Carlin, B.P. (1993).
Maximum-likelihood estimation for constrained- or missing-data
models. 
{\em Canadian Journal of Statistics, \it 21}, 303--311.
\bibitem[Geyer(1999)]{Geye99}
Geyer, C.J. (1999).
Likelihood inference for spatial point processes. In
O.~Barndorff--Nielsen, W.S.~Kendall and M.N.M.~van
Lieshout (Eds.), {\em 
Stochastic geometry, likelihood, and computation}.
CRC Press (pp. 141--172).
\bibitem[Helms(2008)]{Helm08}
Helms, D. (2008).
Temporal analysis. In Gwinn, S.L., Bruce, C., Cooper, J.P. \& Hick, S. (Eds.), {\em Exploring crime analysis: readings on essential skills} (2nd ed.).
International Association of Crime Analysts (pp. 214--257).
\bibitem[Kendall(1998)]{Kend98}
Kendall, W. (1998). Perfect simulation for the area-interaction point process. 
In Accardi, L. \& Heyde, C. (Eds.), 
{\em Probability towards the year 2000}.
Springer (pp. 218--234).
\bibitem[Kendall \& M\o ller(2000)]{KendMoll00}
W.S.~Kendall \& J.~M\o ller (2000).
Perfect simulation using dominating processes on ordered spaces, 
with application to locally stable point processes. 
{\em Advances in Applied Probability,} {\it 32}, 844--865.
\bibitem[Karr(1991)]{Karr91}
Karr, A.F. (1991).
{\em Point processes and their statistical inference\/} (2nd ed.).
Marcel Dekker.
\bibitem[Last \& Brandt(1995)]{Last95}
Last, G. \& Brandt, A. (1995).
{\em Marked point processes on the real line. The dynamic approach\/}.
Springer.
\bibitem[Van Lieshout(1995)]{Lies95}
Lieshout, M.N.M.~van (1995).
{\em Stochastic geometry models in image analysis and spatial statistics}. 
CWI Tracts.
\bibitem[Van Lieshout(2000)]{Lies00}
Lieshout, M.N.M.~van (2000).
{\em Markov point processes and their applications\/}.
Imperial College Press.
\bibitem[Van Lieshout(2016)]{Lies16}
Lieshout, M.N.M.~van (2016).
Likelihood based inference for partially observed renewal processes.
{\em Statistics and Probability Letters, 118}, 190--196.
\bibitem[Van Lieshout \& Baddeley(2002)]{LiesBadd02}
Lieshout, M.N.M.~van \& Baddeley, A.J. (2002).
Extrapolating and interpolating spatial patterns. In Lawson, A.B. and Denison, D.G.T. (Eds.), {\em Spatial cluster modelling}. CRC Press (pp. 61--86).
\bibitem[Mengersen \& Tweedie(1996)]{MengTwee96}
Mengersen, K.L. \& Tweedie, R.L. (1996).
Rates of convergence of the Hastings and Metropolis algorithms.
{\em Annals of Statistics,  24}, 101--121.
\bibitem[Meyn \& Tweedie(2009)]{MeynTwee09}
Meyn, S. \& Tweedie, R.L. (2009).
{\em Markov chains and stochastic stability\/} (2nd ed.).
Cambridge University Press.
\bibitem[Mulder, Ruiter \& Klugkist(2018)]{Muld18} 
Mulder, K., Ruiter, S. \& Klugkist, I. (2018).
Analysis of circular interval-censored data motivated by aoristic 
data in criminology.
{\em Proceedings CMStatistics\/}.
\bibitem[M\o ller \& Waagepetersen(2003)]{MollWaag03}
M\o ller, J. \& Waagepetersen, R.P. (2003).
{\em Statistical inference and simulation for spatial point processes\/}.
CRC Press.
\bibitem[Propp \& Wilson(1996)]{ProppWilson96}
Propp, J.G. \& Wilson, D.B. (1996).
Exact sampling with coupled Markov chains and applications to statistical mechanics. 
{\em Random Structures and Algorithms, \it 9}, 223--252.
\bibitem[Ratcliffe \& McCullagh(1998)]{RatcMcCu98}
Ratcliffe, J.H. \& McCullagh, M.J. (1998).
Aoristic crime analysis.
{\em International Journal of Geographical Information Science, 12}, 751--764.
\bibitem[Ripley \& Kelly(1977)]{RiplKell77}
Ripley, B.D. \& Kelly, F.P. (1977). 
Markov point processes. 
{\em Journal of the London Mathematical Society, 15}, 188--192.
\bibitem[Roberts \& Smith(1994)]{RobeSmit94}
Roberts, G.O. \& Smith, A.F.M. (1994).
Simple conditions for the convergence of the Gibbs sampler
and Metropolis--Hastings algorithms.
{\em Stochastic Processes and their Applications, \it 49}, 207--216.
\bibitem[Ross(1996)]{Ross96}
Ross, S.M. (1996).
{\em Stochastic processes\/} (2nd ed.).
Wiley.
\bibitem[Widom \& Rowlinson(1970)]{WidoRowl70}
Widom, B. \& Rowlinson, J.S. (1970).
A new model for the study of liquid-vapor phase transitions.
{\em Journal of Chemical Physics, 52}, 1670--1684.

\end{thebibliography}

\end{document}